\RequirePackage[l2tabu, orthodox]{nag}%check for outdated package/commands
\documentclass[a4paper]{amsart}%[12pt,a4paper]{amsart}%
\usepackage{mathrsfs}
\usepackage{amsfonts}
\usepackage{amsmath}
\usepackage{amssymb}
\usepackage{amsthm}%"Proof" is defined in amsthm
\usepackage{enumerate}
\usepackage{color}
\usepackage[all]{xy}
\usepackage{geometry}% or [top=0.8in, bottom=1in, left=2cm, right=2cm]{geometry}
\usepackage{setspace}

\usepackage[pdftex,linkcolor=blue,citecolor=blue,backref=page]{hyperref}
\allowdisplaybreaks
\newtheorem{theorem}{Theorem}[section]
\newtheorem{definition}[theorem]{Definition}

\newtheorem{lemma}[theorem]{Lemma}
\newtheorem{corollary}[theorem]{Corollary}
\newtheorem{proposition}[theorem]{Proposition}
\newtheorem{example}[theorem]{Example}
\newtheorem{remark}[theorem]{Remark}
\newtheorem{convention}[theorem]{Convention}

\newcommand{\GK}{\operatorname{GKdim}}
\newcommand{\Aut}{\operatorname{Aut}}

\newcommand{\gr}{\operatorname{gr}}
\newcommand{\id}{\operatorname{id}}

\newcommand{\oo}{{\operatorname{o}}}

\newcommand{\SM}{\operatorname{SM}}

\newcommand{\NNN}{\mathbb{N}}
\newcommand{\RR}{\mathbb{R}}
\newcommand{\ZZ}{\mathbb{Z}}

\newcommand{\fg}{finitely generated}
\newcommand{\fd}{finite dimensional}
\newcommand{\dd}{differential difference}

\newcommand{\GL}{\operatorname{GL}}

\newcommand{\Z}{\operatorname{Z} }%centre

\title[Growth of GWAs]{Growth of generalized Weyl algebras over polynomial algebras and Laurent polynomial algebras\\
\ \\
\textrm{\scriptsize%\footnotesize
Dedicated to Prof. Yuqun Chen on the Occasion of His 65th Birthday}
}
\author[Zhao]{Xiangui Zhao}
\address{School of Mathematics and Statistics,
Huizhou University, Huizhou, Guangdong 516007, China}
\email{zhaoxg@hzu.edu.cn}
\date{\today}   % Activate to display a given date or no date

\begin{document}
\begin{abstract}
We mainly study the growth and Gelfand-Kirillov dimension (GK-dimension) of generalized Weyl algebra (GWA) $A=D(\sigma,a)$ where $D$ is a polynomial algebra or a Laurent polynomial algebra.
  Several necessary and sufficient conditions for $\operatorname{GKdim}(A)=\operatorname{GKdim}(D)+1$ are given.
  In particular, we prove a dichotomy of the GK-dimension of GWAs over the polynomial algebra in two indeterminates,
  namely, $\operatorname{GKdim}(A)$ is either $3$ or $\infty$ in this case.
  Our results generalize several ones in the literature and
  can be applied to determine the growth, GK-dimension, simplicity,
and cancellation properties of some GWAs.
\end{abstract}
%%%%%%%%%%%%
\makeatletter
\@namedef{subjclassname@2020}{%
  \textup{2020} Mathematics Subject Classification}
\makeatother
\subjclass[2020]{
%15A09, %Matrix inversion, generalized inverses
%13P10,% Groebner bases; other bases for ideals and modules (e.g., Janet and border bases)
%16Z10 Groebner-Shirshov bases
16P90, %Growth rate, Gelfand-Kirillov dimension
16S36, %Ordinary and skew polynomial rings and semigroup rings [See also 20M25]
16P40, %Noetherian rings and modules (associative rings and algebras)
16S32. % Rings of differential operators [See also 13N10, 32C38]
}

\keywords{Gelfand-Kirillov dimension,
generalized Weyl algebra,
polynomial automorphism}
\maketitle
%%%%%%%%%%%
%\setcounter{section}{-1}
\section{Introduction and main results}
Throughout let $k$ be a field.
All algebras in consideration are associative unital $k$-algebras and all automorphisms of algebras are $k$-automorphisms.
Let $D$ be an algebra,
$a$ be a central element of $D$ and
$\sigma$ be an automorphism of $D$.
The \emph{generalized Weyl algebra} (GWA, for short)
 $A=D(\sigma,a)$ over $D$ is defined as the algebra generated by $D$ and two indeterminates $x$ and $y$ subject to the relations:
\[
xd=\sigma(d)x,\ yd=\sigma^{-1}(d)y, \
yx=a,\ xy=\sigma(a), \
\text{ for all }d\in D.
\]
GWAs are natural generalizations of Weyl algebras and
form an important class of noncommutative algebras.
Many algebras of common interest fall into this class,
for instance, Weyl algebras,
certain quantum enveloping algebras, some iterated Ore extensions,
noetherian (generalized) down-up algebras, Heisenberg algebras,
 etc. (see \cite{Bavula1992Generalized,bavula2017quantum,Zhao2018Gelfand} for more examples).
GWAs were introduced by Bavula \cite{Bavula1992Generalized} in early 1990s.
Since then
their structures and representations have been studied intensively, for example,
their dimension and growth \cite{bavula1996global,Bavula1998Krull,bavula2001krull,Zhao2018Gelfand};
their derivations, isomorphisms and automorphisms
\cite{bavula2001isomorphism,Suarez-Alvarez_2015,Almulhem2018,Gaddis2019,Gutierrez2020};
and their homological and geometric properties \cite{Brzezinski2016,Won2018,Liu2018,Ferraro2020,liu2021batalin,klyuev2021twisted}.

The notion of growth,
introduced by Gelfand and Kirillov \cite{gelfand1966corps}
for algebras
and by Milnor \cite{milner1968curvature} for groups,
is a fundamental object of study in theory of algebras (not necessarily associative) and groups.
The Gelfand-Kirillov dimension (GK-dimension, for short) is a main tool to quantify growth of algebras and groups.
The GK-dimension of an algebra $A$ is defined as
\[
\GK(A):=\sup_V{\limsup_{n\to\infty}}\log_n \dim_k\left(\sum_{i=0}^nV^i\right)
\]
where the supremum is taken over all finite dimensional subspaces $V$ of $A$.
GK-dimension basically measures the growth of an algebra (rather than the growth of a ring, see \cite{Wu1991}).
For a finitely generated commutative algebra,
its Krull dimension and GK-dimension coincide.
Thus GK-dimension can be viewed as a noncommutative analogue of Krull dimension.
It turns out that the GK-dimension is a very useful and powerful tool for investigating noncommutative algebras,
for example, in recent studies of classification problems
(see e.g. \cite{Wang2015,GOODEARL2017,Andruskiewitsch2020})
and cancellation problems (see e.g. \cite{Bell2017Zariski,Lezama2018,tang2020cancellation,chan2022reflexive}).
For basic properties and applications of GK-dimension of algebras and groups we refer the reader to \cite{Krause-Lenagan_2000}.
For GK-dimension of other algebraic structures, see e.g.
\cite{Khoroshkin2015,Bao2020,Bai2021,2020Growth}.

There have been in the literature
a number of results concerning GK-dimension of skew polynomial extensions related to derivations and/or automorphisms,
for example, the GK-dimensions of Ore extensions \cite{lorenz1982gelfand, huh1996gelfand,zhang1997note},
 of PBW-extensions \cite{matczuk1988gelfand},
 and of \dd\ algebras/modules \cite{ZhangZhao2013algebras}.
Recently, the author \cite{Zhao2018Gelfand} investigated general properties of GK-dimensions of GWAs.

In this paper, we mainly study the growth and GK-dimension of GWAs
over polynomial algebras $P_n:=k[z_1,z_2,\dots,z_n]$
(especially $P_2$)
and that over Laurent polynomial algebras $L_n:=k[z_1^{\pm1},z_2^{\pm1},\dots,z_n^{\pm1}]$.
We have two motivations for this paper.
The first motivation is the fact that
lots of GWAs studied in literature are ones over $P_n$ or $L_n$
and thus this subclass of GWAs is of fundamental importance in practice.
For instance, (quantum) Weyl algebras, quantum planes, and primitive quotients of the quantum enveloping algebra $\mathcal{U}_q(\mathfrak{sl}_2)$ of special linear Lie algebra
$\mathfrak{sl}_2$ are GWAs over $P_1$ or $L_1$;
noetherian generalized down-up algebras, quantum Heisenberg algebra,
and the universal enveloping algebra $U(\mathfrak{sl}_2)$
are GWAs over $P_2$;
the group algebra of the {(discrete) Heisenberg group}
(Example \ref{exam_Heisenberg}) and
the homogenized enveloping algebra of some Lie algebra
(\cite[Example 2.4 (ii)]{jordan2009reversible})
are GWAs over $L_2$.
Our second motivation is that,
since
the automorphism groups $\Aut(P_2)$ and $\Aut(L_n)$ are well-understood
(see Jung and van der Kulk \cite{Jung1942Uber,vanderKulk1953} for $\Aut(P_2)$ and Theorem \ref{thm_auto(L)} for $\Aut(L_n)$),
sharper theorems can be expected when one investigates the GK-dimension of GWAs over $P_2$ and $L_n$.

Let $A=D(\sigma,a)$ be a GWA.
It was proved in \cite{Zhao2018Gelfand} that the difference
$\GK(A)-\GK(D)$ can be any positive integer or infinity.
If $\sigma$ is \emph{locally algebraic} (Definition \ref{def-la}) meaning that
every \fd\ subspace of $D$ is contained in a \fd\ $\sigma$-stable subspace of $D$,
then $\GK(A)-\GK(D)=1$.
Although $\GK(A)-\GK(D)=1$ does not imply in general that $\sigma$ is locally algebraic (see Example \ref{exam-GK+1-NotLA}),
we prove that, for GWAs over $P_n$ and $L_n$, this implication does hold.
More generally,
we have the following theorem.

\begin{theorem}
  \label{thm-main-}
  Let $A=D(\sigma,a)$ be a GWA.
  Let $P_n=k[z_1,z_2,\dots,z_n]$ be the polynomial algebra in $n$ indeterminates and $Q_n$ be the fraction field of $P_n$.

  If $P_n\subseteq D\subseteq Q_n$,
  then the following statements are equivalent:
  \begin{enumerate}[(a)]
    \item\label{itemD+1}
    $\GK(A)=\GK(D)+1$.
    \item\label{itemD+LA}
    $\sigma$ is locally algebraic.
  \end{enumerate}

  If $k$ is algebraically closed, $D$ is a \fg\ field and $a\neq0$,
 then (\ref{itemD+1}) and (\ref{itemD+LA}) are equivalent to each of the following statements:
 \begin{enumerate}[(a)]
 \setcounter{enumi}{2}
    \item\label{itemD-FinOrd}
     $\sigma$ has a finite order.
     \item\label{itemD-NotSimple}
     $A$ is not simple.
 \end{enumerate}
\end{theorem}

Theorem \ref{thm-main-} is an analogue of \cite[Theorem 1.1]{zhang1997note} but note that we do not assume that
$k$ is algebraically closed to get $(a)\Leftrightarrow (b)$.
See Proposition \ref{prop-Zhang97} for more results for the algebraically closed case.

When considering GWAs over $P_2$,
we have more elaborate results.
Our main results on  the GK-dimension of GWAs over $P_2$ can be summarized as
\begin{theorem}
    \label{thm-mainP}
  Let $A=k[z_1,z_2](\sigma,a)$ be a GWA.
     Then the following statements are equivalent:
  \begin{enumerate}[(a)]
    \item\label{itemP+1}
    $\GK(A)=3$.
    \item\label{itemP-GKfinite}
    $\GK(A)<\infty$.
    \item\label{itemP-PolyGrowth}
    $A$ has polynomial growth.
    \item\label{itemP-NotExGrowth}
    $A$ does not have exponential growth.
    \item\label{itemP-LA}
    $\sigma$ is locally algebraic.
    \item\label{itemP-ConjTri}
    $\sigma$ is conjugate to a triangular automorphism.
    \item\label{itemP-PowerConjTri}
    $\sigma^m$ is conjugate to a triangular automorphism for some $m\geq 1$.
  \end{enumerate}
\end{theorem}

The growth of algebras (Definition \ref{def-growth}) can be prescribed within a rather wide range of functions.
Recently,
Bell and Zelmanov \cite{Bell2020} give a complete characterization of the functions that can occur as the growth functions of algebras.
Theorem \ref{thm-mainP} indicates that
GWAs over $P_2$ have \emph{alternative growth} (Definition \ref{def-growth}).
Note that not every algebra or group has alternative growth,
for example,
if a \fg\ Lie
algebra $\mathfrak{g}$ has polynomial growth then its universal enveloping algebra $U(\mathfrak{g})$ has intermediate
growth \cite{Smith1976}.
It is still a major open question in growth theory whether there exist finitely presented groups with intermediate growth \cite{Grigorchuk2008}.

The dichotomy of the GK-dimension of GWAs over $P_2$
(i.e., $\GK(P_2(\sigma,a))$ is either $3$ or $\infty$)
does not hold for that over $L_2$ (Example \ref{exam_Heisenberg}).
Note that some items in Theorem \ref{thm-mainP} are not equivalent for GWAs over $P_n$ with $n\geq 3$.
For example, the Nagata automorphism of $P_3$ is locally algebraic but
not conjugate to a triangular automorphism (Example \ref{exam-Nagata}).
The reason for this is that
the proof of Theorem \ref{thm-mainP} is
based on the famous theorem of Jung and van der Kulk \cite{Jung1942Uber,vanderKulk1953} asserting
that all automorphisms of $P_2$ are tame,
which is not true for $P_n$ with $n\geq3$ (\cite{Shestakov2003}).
See Remark \ref{remark-P2-P3} for more details.

We also study further the GK-dimension of GWAs over $L_n$.
An automorphism $\sigma\in\Aut(L_n)$ sends a unit to a unit and thus has the form
\begin{align*}
\sigma(z_i)=\alpha_iz_1^{a_{1i}}z_2^{a_{2i}}\cdots z_n^{a_{ni}},
 \ \alpha_i\in k^*, a_{1i},a_{2i},\dots, a_{ni}\in\ZZ,\
i=1,2,\dots,n.
\end{align*}
The automorphism $\sigma$ is determined by the matrix
\(
M_{\sigma}=(a_{ij})\in \ZZ^{n\times n}
\)
and the scalars $\alpha_i\in k^*$, $i=1,2,\dots,n$.
The matrix $M_{\sigma}$ plays an important role in the
GK-dimension of the GWA $L_n(\sigma,a)$.
Our main results on the
GK-dimension of GWAs over $L_n$ are summarized in the following theorem.
\begin{theorem}
  \label{thm-mainL}
  Let $A=k[z_1^{\pm1},z_2^{\pm1},\dots,z_n^{\pm1}](\sigma,a)$ be a GWA.
  Then the following statements are equivalent.
  \begin{enumerate}[(a)]
    \item\label{itemL+1}
    $\GK(A)=n+1$.%\footnote{How to tell if $\GK=\infty$?}
    \item\label{itemL-LA}
    $\sigma$ is locally algebraic.
    \item\label{itemL-FiniteOrd}
    $M$ has a finite order, i.e., $M_{\sigma}^m$ equals the identity matrix for some $m\geq1$.
    \item\label{itemL-PowerBound}
    $M_{\sigma}$ is power-bounded, i.e.,
    there exists $N>0$ such that the absolute values of all entries of $M_{\sigma}^m$ are less than $N$ for all $m\geq1$.
  \end{enumerate}
\end{theorem}

Theorem \ref{thm-mainL} enables us to determine the GK-dimension by computation of integer matrices under certain conditions.
Given $n\geq 1$,
Levitt \cite{Levitt1998} showed that the maximum order $G(n)$ of a matrix in $\GL(n,\ZZ)$ satisfies a Landau-type estimate $\ln G(n)\sim \sqrt{n\ln n}$
and all possible orders of matrices in $\GL(n,\ZZ)$ can be determined by factorization of natural numbers less than $G(n)$  (see Remark \ref{remark-MaxOrd}).

Our results can be applied to determine the GK-dimension, the simplicity (Corollary \ref{coro-Simple}),
and cancellation properties (Corollary \ref{coro-Cancellation}) of some GWAs.

The rest of this paper is organized as follows.
In Section \ref{sec_preliminaries},
we recall basic definitions and list some results from the literature.
In Section \ref{sec_SMC},
we study the sensitive multiplicity condition and
prove Theorem \ref{thm-main-}.
In Section \ref{sce-Poly},
we focus on the growth and GK-dimension of GWAs over polynomial algebras,
where Theorem \ref{thm-mainP} is proved.
In Section \ref{sce-Laurent},
 we first investigate automorphisms of Laurent polynomial algebras,
 then study the GK-dimension of GWAs over Laurent polynomial algebras, and finally prove Theorem \ref{thm-mainL}.

\section{Preliminaries}\label{sec_preliminaries}
In this section,
we recall some terminology related to generalized Weyl algebras and Gelfand-Kirillov dimension.
Our main references for GK-dimension are \cite{Krause-Lenagan_2000} and \cite[Chapter 8]{McConnell-Robson2001}.

Let $k$ be a field and $k^*=k\setminus \{0\}$.
Unless otherwise stated,
all spaces in this article are $k$-spaces,
$\dim(V)$ denotes the dimension of a $k$-space $V$,
all algebras are associative unital $k$-algebras,
and
all automorphisms are $k$-algebra automorphisms.
Denote the space spanned by a set $S$ by $kS$,
particularly by $ka$ if
$S=\{a\}$.
Given subspaces $V$ and $W$ of an algebra $D$,
denote $VW=\sum_{v\in V, w\in W}k(vw)$,
$V^0=k$,
$V^n=V^{n-1}V$ for $n\geq 1$,
and write $VW=Va$ if $W=ka$.
We use $\Z(D)$ (respectively, $\Aut(D)$, $\id=\id_D$) to denote the center (respectively, automorphism group, identity automorphism) of $D$.

\begin{definition}[\cite{Bavula1992Generalized}]
  \label{def-GWA}
Let $D$ be an algebra,
$\sigma\in \Aut(D)$, and $a\in \Z(D)$.
The \emph{generalized Weyl algebra} (GWA, for short)
$A=D(\sigma,a)=D(x,y;\sigma,a)$ with the \emph{base algebra} $D$ is defined as the algebra generated by $D$ and indeterminates $x$ and $y$, subject to the defining relations:
\begin{align*}
  \label{eqn_def-GWA-n}
 xd=\sigma(d)x, \ yd=\sigma^{-1}(d)y,\
 yx=a, \ xy=\sigma(a), \text{ for all }d\in D.
 \end{align*}
The automorphism $\sigma$ and element $a$ are called \emph{defining automorphism} and \emph{defining element} of the GWA $A$ respectively.
\end{definition}

The following example indicates why the algebra $D(\sigma,a)$ is called a GWA.
\begin{example}
  \label{exam_WeylAlg}\upshape
 The \emph{first Weyl algebra} $A_1$ over $k$ is the ring of polynomials in indeterminates $x$ and $y$ with coefficients in $k$, subject to the relations
  $yx-xy=1$.
 Then
 \[
 A_1\cong D(\sigma,a), \ x\mapsto x, \ y\mapsto y,
 \]
  where $D=k[h]$, $a=h=yx$ and $\sigma(h)=h-1$.
 The $n$-th Weyl algebra can be presented as an iterated GWA. See \cite{Zhao2018Gelfand} for more detailed examples of GWAs.
\end{example}

A \emph{$\ZZ$-filtration} of algebra $A$ is a sequence of subspaces
  \[
  \cdots\subseteq F_{i-1}\subseteq  F_{i}\subseteq F_{i+1}\subseteq \cdots, i\in\ZZ
  \]
  such that $1\in F_0$,
  $F_iF_j\subseteq F_{i+j}$ for all $i,j\in\ZZ$,
  and $A=\bigcup_{i\in\ZZ}F_i$.
  An algebra with a $\ZZ$-filtration is called a $\ZZ$-filtered algebra.
  The vector space $\gr(A)=\oplus_{i\in\ZZ} F_i/F_{i-1}$
  equipped with the linear multiplication defined by the rule
  \[
  (a+A_{i-1})(b+F_{j-1})=ab+F_{i+j-1}
  \]
  is called the \emph{the associated graded algebra} of $A$.
  The linear mapping $\gr:A\to \gr(A)$ is defined by
  $\gr(a)=a+F_{i-1}$ for all $a\in F_i/F_{i-1}$.
  %is called a \emph{leading term} map.
  For a subspace $V\subseteq A$,
  denote $\gr(V)=\sum_{v\in V}k\gr(v)$.
  The following lemma lists some basic properties of the map $\gr$.
  For more properties related to filtered algebras and graded algebras, see \cite[Chapter 6]{Krause-Lenagan_2000}.

\begin{lemma}[\cite{zhang1997note}, Lemma 2.1]
  \label{lemma-LeadTermMap}
  Let $A$ be a $\ZZ$-filtered algebra, $W$ be a \fd\ subspace of $A$, and
  $a\in A$.
  Then
\begin{enumerate}[(i)]
  \item
  $\gr(Wa)\supseteq \gr(W)\gr(a)$.
  \item
  $\dim(W)=\dim(\gr(W))$.
  \item
  $\dim(W^m)\geq\dim((\gr(W))^m)$.
\end{enumerate}
\end{lemma}

The GWA $A=D(x,y;\sigma,a)$ is a free left $D$-module with a basis
$\{x^i,y^j:i\geq 0, j>0\}$.
There is a natural
 $\ZZ$-graded structure for a GWA, i.e.,
$D(\sigma,a)=\bigoplus_{i\in \ZZ}A_i$, where $A_i=Dx^i$ if $i>0$,
$A_i=Dy^{-i}$ if $i<0$, and $A_i=D$ if $i=0$.
Let $F_i=\cup_{j\leq i}A_j$.
Then
$$
  \cdots\subseteq F_{i-1}\subseteq  F_{i}\subseteq F_{i+1}\subseteq \cdots, i\in\ZZ
$$
is a $\ZZ$-filtration for $A$.
Recall that, given $\sigma\in\Aut(D)$,
the Ore extension $D[x;\sigma]$ over an algebra $D$ is the algebra generated by $D$ and $x$ subject to the relation $xd=\sigma(d)x$ for all $d\in D$.
The following lemma follows immediately from the $\ZZ$-graded structure of GWAs.

\begin{lemma}
  \label{lemma_Ore-subalg}
The Ore extension $D[x;\sigma]$ is a subalgebra of the GWA $D(\sigma,a)$.
\end{lemma}

It is routine to check that GWAs have the following universal property (cf. \cite[Chapter 1.2.5]{McConnell-Robson2001}).

\begin{proposition}
  \label{prop-universal}
  Let $D(x,y;\sigma,a)$ be a GWA and $B$ be an algebra.
  Suppose $\phi:D\to B$ is an algebra homomorphism,
  and $x',y'\in B$ satisfy
  \[
  x'\phi(d)=\phi(\sigma(d))x', y'\phi(d)=\phi(\sigma^{-1}(d))y', y'x'=\phi(a), x'y'=\phi(\sigma(a)), \text{ for all } d\in D.
  \]
  Then there exists a unique algebra homomorphism
  $\psi:D(x,y;\sigma,a)\to B$ such that $\psi(x)=x'$, $\psi(y)=y'$ and the following diagram
  \begin{align*}
  \xymatrix{
  D \ar[rr]^{} \ar[dr]_{\phi}
                &  &    D(x,y;\sigma,a) \ar[dl]^{\psi}    \\
                & B                 }
  \end{align*}
is commutative.
\end{proposition}

Let $\Phi$ denote the set of all eventually monotone increasing functions $f:\NNN\to \RR^+$.
For $f,g\in\Phi$ set $f\leq^* g$ if and only if there exist $c,m,N\in\NNN$ such that
\[
f(n)\leq cg(mn) \text{ for all } n>N
\]
and set
\[
f\sim g\ \Leftrightarrow
f\leq^*g \text{ and }g\leq^*f.
\]
With a slight abuse of notation,
we often write $f(n)\sim g(n)$ to mean that $f\sim g$.
One can check that $\sim$ is an equivalent relation on $\Phi$.
For $f\in\Phi$ the equivalence class $\mathcal{G}(f)\in \Phi/\sim$ is called the \emph{growth of function $f$}.
At the risk of abusing the notation,
we also write the growth of $f$ as $\mathcal{G}(f(n))$.
The partial order on $\Phi/\sim$ induced by $\leq^*$ is denoted by $\leq $.

If an algebra $A$ is generated (as an algebra) by a subspace $V$,
then $V$ is called a \emph{generating subspace} of $A$.
If $V$ and $W$ are two \fd\ generating subspaces of $A$,
then $\mathcal{G}(d_V)=\mathcal{G}(d_W)$ (\cite[Lemma 1.1]{Krause-Lenagan_2000}),
where $d_V(n)=\dim\left(\sum_{i=0}^nV^i\right)$ for all $n\in\NNN$.
Thus we can define the  notion of growth of algebras,
which is independent from the choice of generating subspaces.

\begin{definition}
  \label{def-growth}
  Let $A$ be a \fg\ algebra.
  \begin{enumerate}[(i)]
    \item
    Suppose $A$ is
  generated by a \fd\ subspace $V$.
  Then 
  \[
  d_V(n):=\dim\left(\sum_{i=0}^nV^i\right), \ n\in\NNN
  \]is called the \emph{growth function} of $A$ with respect to $V$ and
  $\mathcal{G}(A):=\mathcal{G}(d_V)$ is called the \emph{growth of $A$}.
  \item
  A \fg\ algebra $A$ is said to have
  \begin{itemize}
    \item
    \emph{polynomial growth} if $\mathcal{G}(A)=\mathcal{G}(n^d)$ for some $d\in\NNN$;
    \item
    \emph{exponential growth} if
  $\mathcal{G}(A)=\mathcal{G}(e^n)$;
    \item
    \emph{intermediate growth} if
  $\mathcal{G}(A)<\mathcal{G}(e^n)$
  yet $\mathcal{G}(A)\not\leq\mathcal{G}(n^m)$ for all $m\in\NNN$.
  \item
  \emph{alternative growth} if $A$ has either polynomial growth or exponential growth.
  \end{itemize}
  \end{enumerate}
\end{definition}

Functions within a wide range can be realized as
growth functions of \fg\ algebras.
Particularly, exponential growth is the largest one that an algebra can have.
Recently,
Bell and Zelmanov \cite{Bell2020} give a complete characterization of the functions that can occur as the growth functions of algebras.

Gelfand-Kirillov dimension is a numerical characteristic of growth, which works for not only \fg\ algebras but also infinitely generated algebras.

\begin{definition}
  \label{def-GKdim}
  Let $A$ be an algebra.
The \emph{Gelfand-Kirillov dimension} (\emph{GK-dimension}, for short) of $A$ is defined as
\[
\GK(A):=\sup_V{\limsup_{n\to\infty}}\log_n d_V(n)
\]
where the supremum is taken over all finite dimensional subspaces $V$ of $A$.
\end{definition}

A \fd\ subspace $V\subseteq A$ is called a \emph{subframe} of $A$ if $1\in V$.
If $A$ is generated by a subframe $V$ (and thus $A$ is \fg),
then
\[
\GK(A)={\limsup_{n\to\infty}}\log_n \dim(V^n).
 \]

The following well known fact is useful in the computation with growth functions and GK-dimension. %The proof is straightforward and is omitted.

\begin{lemma}
  \label{lemma-GK-sum}
  If $f(n)\in\Phi$ is a polynomial of degree $d\in\NNN^*$,
  then $g(n):=f(0)+f(1)+\cdots+f(n)$ is a polynomial of degree $d+1$.
 \end{lemma}

The following lemma provides useful properties of GWAs,
which will be used later.
\begin{lemma}[Lemma 2.7 of \cite{bavula2001isomorphism}]
  \label{lemma_iso}
Let $A=D(\sigma,a)$ be a GWA and let $\tau\in\Aut(D)$.
Then
\begin{enumerate}[(i)]
\item
 $A\cong D(\tau^{-1}\sigma\tau,\tau^{-1}(a))$.
\item
$A\cong D(\sigma^{-1},\sigma(a))$.
\item
If $\lambda $ is a central unit in $D$, then $A\cong D(\sigma; a\lambda)$.
In particular, if $a$ is a unit then $A\cong D(x^{\pm1};\sigma,1)= D[x^{\pm1};\sigma]$, the twisted Laurent polynomial algebra over $D$.
\end{enumerate}
\end{lemma}
\begin{proof}
  The isomorphisms are given respectively by (i) $x\mapsto x$, $y\mapsto y$ and $d\mapsto \tau^{-1}(d)$ for all $d\in D$;
  (ii) $x\mapsto y$, $y\mapsto x$ and $d\mapsto d$ for all $d\in D$;
  (iii) $x\mapsto x\lambda^{-1}, y\mapsto y$ and $d\mapsto d$ for all $d\in D$.
\end{proof}

\begin{definition}\upshape
\label{def-la}
Suppose $D$ is an algebra and $\sigma,\tau\in\Aut(D)$.
\begin{enumerate}[(i)]
  \item
  $\tau$ is \emph{conjugate} to $\sigma$ if
  $\tau=\eta^{-1} \sigma \eta$ for some $\eta\in \Aut(D)$.
  \item
  A subspace $V$ of $D$ is said to be \emph{$\sigma$-stable} if $\sigma(V)=V$.
  \item
  $\sigma$ has a \emph{finite order} if $\sigma^m=\id$ for some $m\geq1$.
  \item
  $\sigma$ is \emph{inner} if there exists a unit $u\in D$ such that
  $\sigma(d)=udu^{-1}$ for all $d\in D$.
  \item
  $\sigma$ is \emph{locally algebraic}
  %or \emph{locally finite} (\cite{zhang1997note,Leroy1988On})
  if every \fd\ subspace of $D$  is contained in a \fd\ $\sigma$-stable subspace of $D$.
\end{enumerate}
\end{definition}

\begin{lemma}
  \label{lemma-LocAlgIff}
  An automorphism $\sigma\in\Aut(D)$ is locally algebraic if and only if
  for every $a\in D$ the set $\{\sigma^i(a): i\in\NNN\}$
is contained in a \fd\ subspace of $D$.
\end{lemma}
\begin{proof}
  ($\Rightarrow$) Suppose $\sigma$ is locally algebraic.
  Then, for every $a\in D$, the subspace $k\{a\}$ is \fd\
  and thus contained in a \fd\ $\sigma$-stable subspace $U$.
  As a result $\{\sigma^i(a): i\in\NNN\}\subseteq U$.

  ($\Leftarrow$)
  Suppose $V=k\{a_1,a_2,\dots,a_m\}$ is a \fd\ subspace of $D$.
  By the assumption, each space $k\{a_i\}$ is contained in a \fd\
  $\sigma$-stable subspace $U_i\subseteq D$, $1\leq i\leq m$.
  Then $V\subseteq U:=U_1+U_2+\dots+U_m$.
  Clearly $U$ is \fd\ and $\sigma$-stable.
\end{proof}

Locally algebraicity has a close relation with the GK-dimension of algebras extended by automorphisms.

\begin{lemma}[\cite{Zhao2018Gelfand}, Lemma 3.1 and Theorem 3.4]
  \label{lem_GK-CA}
  Let $A=D(\sigma,a)$ be a GWA.
  Then
  \begin{enumerate}[(i)]
    \item
    $\GK(A)\geq \GK(D)+1$.
    \item
    $\GK(A)=\GK(D)+1$ if moreover $\sigma$ is locally algebraic.
  \end{enumerate}

\end{lemma}

\begin{lemma}[\cite{zhang1997note},Theorem 1.1]
  \label{lemma-zhang97}
  Let $D$ be a commutative domain over an algebraically closed field $k$ such that the fraction field of $D$ is finitely generated as a field,
  and let $\sigma\in\Aut(D)$.
  Then the following statements are equivalent:
  \begin{enumerate}[(i)]
    \item $\GK(D[x,\sigma])=\GK(D)+1$.
    \item
    $\sigma$ is locally algebraic.
  \end{enumerate}
 If moreover $D$ is a field,
 then (i) and (ii) are equivalent to the following statement:
 \begin{enumerate}
    \item[(iii)]
     $\sigma$ has a finite order.
 \end{enumerate}
\end{lemma}

We conclude this section by two lemmas
%a nice criterion of the simplicity
of GWAs due to Bavula.
\begin{lemma}[\cite{Bavula1992Generalized}]
  \label{lemma-Noeth}
Let $A=D(\sigma,a)$ be a GWA.
  Suppose $a\neq0$.
  If $D$ is Noetherian (respectively, without zero divisors),
then $A$ is Noetherian (respectively, without zero divisors).
\end{lemma}

\begin{lemma}[\cite{Bavula1996}, Theorem 4.2]
  \label{lemma-Simplicity}
  Let $A=D(\sigma,a)$ be a GWA.
  Then $A$ is simple if and only if the following conditions hold:
  \begin{enumerate}[(i)]
    \item
    $D$ has no proper ideals that are $\sigma$-stable;
    \item
    $\sigma^m$ is not inner for all $m\geq 1$;
    \item
    $Da+D\sigma^i(a)=D$ for all $i\in\NNN$;
    \item
    $a$ is regular in $D$, i.e., $ad\neq 0$ and $da\neq 0$ for all $0\neq d\in D$.
  \end{enumerate}
\end{lemma}

\section{Sensitive multiplicity condition}
\label{sec_SMC}
Let
$P_n=k[z_1,z_2,\cdots,z_n]$ and
$Q_n$ be the fraction field of $P_n$.
In this section, we first prove that all algebras $D$ such that
$P_n\subseteq D\subseteq Q_n$ satisfy the sensitive multiplicity condition (Definition \ref{def-SMC}) without assuming that
$k$ is algebraically closed and then we prove Theorem \ref{thm-main-}.

The notion of {sensitive multiplicity condition} was introduced by Zhang \cite{zhang1997note},
and it plays an important role in GK-dimension of skew polynomial extensions.
Recall that $a\in A$ is called a \emph{regular element} in algebra $A$ if
$ab\neq 0$ and $ba\neq 0$ for all $0\neq b\in A$.

\begin{definition}
\label{def-SMC}
  Let $A$ be an algebra of $\GK(A)=d$.
We say $A$ satisfies the \emph{sensitive multiplicity condition}
(SMC for short) $\SM(V_0,c,d)$ if there exist a \fd\ subspace $V_0$ of $A$ and a constant $c>0$ such that
if $W$ is a \fd\ subspace of $A$ containing $V_0a$ for some regular element $a\in A$, then
$\dim(W^m)\geq c\dim(W) m^d$ for all $m\in\NNN$.
\end{definition}

\begin{lemma}[\cite{zhang1997note}, Lemma 3.1(2)]
  \label{lemma-SMCQuotient}
  Let $A$ be a commutative domain and $Q(A)$ be the fraction field of $A$.
  Then $Q(A)$ satisfies $\SM(V_0,c,d)$ if and only if $A$ satisfies
  $\SM(V_0a,c,d)$ for some $0\neq a\in A$ such that $V_0a\subseteq A$.
\end{lemma}
\begin{proof}
  Suppose $Q(A)$ satisfies $\SM(V_0,c,d)$.
  Since $V_0$ is \fd,
  we can write $V_0=Va^{-1}$ for some \fd\ subspace $V\subseteq A$ and $0\neq a\in A$.
  Suppose $W$ is a \fd\ subspace of $A$ such that $W\supseteq Vb$ for some regular element $b\in A$.
  Then $W\supseteq V_0(ab)$ and $ab\in Q(A)$ is regular.
  By SMC for $Q(A)$,
  \begin{align*}
    \dim(W^m)\geq c\dim(W)m^d, \ \ m\in\NNN.
  \end{align*}
  Thus $A$ satisfies $\SM(V_0a,c,d)$ for $a\in A$ regular and $V_0a=V\subseteq A$.

  Suppose $A$ satisfies
  $\SM(V_0a,c,d)$ for some $V_0\subseteq Q(A)$ and $0\neq a\in A$ such that $V_0a\subseteq A$.
  Let $W$ be a \fd\ subspace of $Q(A)$ and
  $V_0q\subseteq W$ for some $0\neq q\in Q(A)$.
  We can write $W=W'b^{-1}$ for some \fd\ subspace $W'\subseteq A$ and $0\neq b\in A$
  such that $a^{-1}qb\in A$.
  Then $(V_0a)(a^{-1}qb)=V_0qb\subseteq Wb=W'$ and thus, by SMC for $A$,
  \begin{align*}
    \dim(W^m)=\dim(W'^m)\geq c\dim(W')m^d=c\dim(W)m^d, \ \ m\in\NNN.
  \end{align*}
  Thus $Q(A)$ satisfies $\SM(V_0,c,d)$.
\end{proof}

\begin{lemma}
  \label{lemma-SMC-GK+2}
  If $D$ is a commutative algebra satisfying SMC and
  $\sigma\in\Aut(D)$ is not locally algebraic,
  then $\GK(D(\sigma,a))\geq \GK(D)+2$.
\end{lemma}
\begin{proof}
  By Lemma \ref{lemma_Ore-subalg}, $D[x,\sigma]$ is a subalgebra of $D(\sigma,a)$.
  Thus, by \cite[Lemma 3.1]{Krause-Lenagan_2000} and \cite[Proposition 3.3 (1)]{zhang1997note}, we have
  \[
  \GK(D(\sigma,a))\geq \GK(D[x,\sigma])\geq \GK(D)+2,
  \]
  which completes the proof.
\end{proof}

Following the proof of \cite[Theorem 3.2]{zhang1997note}, we have the following lemma without the assumption that $k$ is algebraically closed.

\begin{lemma}
  \label{lemma-PnSMC}
   For $n\geq 0$, the polynomial algebra $P_n$ satisfies $\SM(V_n,c_n,n)$
  where
  \[
  V_n=k+kz_1+\cdots+kz_n,\
  c_n=\frac{1}{2^n\sqrt5^{n(n+1)}n!}.
  \]
\end{lemma}
\begin{proof}
The proof of \cite[Theorem 3.2]{zhang1997note} with small modification still works for our lemma.
For the reader's convenience, we include the slightly modified proof in the Appendix.
\end{proof}

\begin{corollary}
  \label{coro-SMCPoly}
  Suppose $D$ is an algebra such that $P_n\subseteq D\subseteq Q_n$.
  Then $D$ satisfies SMC.
\end{corollary}
\begin{proof}
  Let $V=k+kz_1+\cdots+kz_n\subseteq D$.
  Suppose that $W$ is a \fd\ subspace of $D$ such that
  $Va\subseteq W$ for some $0\neq a\in D$.
  Note that $\GK(D)=n$ since $\GK(P_n)=\GK(Q_n)=n$.
  By Lemmas  \ref{lemma-PnSMC} and \ref{lemma-SMCQuotient}, $Q_n$ satisfies $\SM(V_n,c_n,n)$, where $V_n, c_n$ are as in Lemma \ref{lemma-PnSMC}.
  Hence by the definition of SMC
  \[
  \dim(W^m)\geq c\dim(W) m^n \text{ for some $c>0$ and all}\ m\in\NNN,
  \]
  which means that $D$ satisfies SMC.
\end{proof}

Now we have the following result without the assumption that
$k$ is algebraically closed.

\begin{proposition}%[Theorem \ref{thm-main-}]
  \label{prop-Zhang97PL}
  Let $A=D(\sigma,a)$ be a GWA,
   where $P_n\subseteq D\subseteq Q_n$.
  Then the following statements are equivalent:
  \begin{enumerate}[(i)]
    \item $\GK(A)=\GK(D)+1$.
    \item
    $\sigma$ is locally algebraic.
  \end{enumerate}
\end{proposition}

\begin{proof}
  By Corollary \ref{coro-SMCPoly}, $D$ satisfies SMC.
  Then the statement follows from Lemmas \ref{lem_GK-CA} and \ref{lemma-SMC-GK+2}.
\end{proof}

Proposition \ref{prop-Zhang97PL} is a partial converse of Lemma \ref{lem_GK-CA} (ii).
If $k$ is algebraically closed,
we have the following proposition,
which is an analogue of Lemma \ref{lemma-zhang97}.

\begin{proposition}
  \label{prop-Zhang97}
  Let $A=D(\sigma,a)$ be a GWA,
   where $D$ is a commutative domain over an algebraically closed field $k$ such that the fraction field of $D$ is finitely generated as a field.
  Then the following statements are equivalent:
  \begin{enumerate}[(i)]
    \item $\GK(A)=\GK(D)+1$.
    \item
    $\sigma$ is locally algebraic.
  \end{enumerate}
 If moreover $D$ is a field,
 then (i) and (ii) are equivalent to the following statement:
 \begin{enumerate}
    \item[(iii)]
     $\sigma$ has a finite order.
    \end{enumerate}
\end{proposition}

\begin{proof}
  (ii)$\Rightarrow$(i) follows from Lemma \ref{lem_GK-CA}.
  (i)$\Rightarrow$(ii) follows from Lemma \ref{lemma-SMC-GK+2}.

  (ii)$\Leftrightarrow$(iii).
  The proof of Lemma \ref{lemma-zhang97} (see \cite[p.369]{zhang1997note}) still works for our statements.
\end{proof}

Note that (i)$\Rightarrow$(ii) in Proposition \ref{prop-Zhang97} fails if $D$ is not commutative,
see Example \ref{exam-GK+1-NotLA}.

\begin{example}\upshape
  \label{exam-GK+1-NotLA}
  Let $D=k[z^{\pm1},s][d;\sigma]$ be an Ore extension over polynomial ring $k[z^{\pm1},s]:=k[z^{\pm1}][s]$,
  where the automorphism $\sigma$ is defined by
  $\sigma(z)=z$, $\sigma(s)=zs$.
  Let $A=D(x,y;\tau,1)$ be a GWA, where $\tau\in\Aut(D)$ satisfies
  $\tau(z)=z$, $\tau(s)=s$, $\tau(d)=z^{-1}d$.
  Then $\tau$ is not locally algebraic but
  $\GK(A)=\GK(D)+1=5$.
\end{example}

\begin{proof}
  First note that $D$ is the differential difference algebra defined in \cite[Example 3.12]{ZhangZhao2013algebras} and
  $\GK(D)=4$ (see \cite[p.494]{ZhangZhao2013algebras}).
  By Lemma \ref{lem_GK-CA},
  $\GK(A)\geq \GK(D)+1=5$.
  It remains to show that
  $\GK(A)\leq 5$.
  Rewrite $A$ as an Ore extension over $C:=k[x^{\pm1},z^{\pm1},s]$,
  that is,
  $A=C[d;\delta]$,
  where $\delta\in\Aut(C)$ is defined by
  $\delta(z)=z$, $\delta(x)=zx$, $\delta(s)=zs$.
  Let $V=k\{x^{\pm1},z^{\pm1},s,1\}$.
  Then $W=V+kd$ is a \fd\ generating subspace of $A$.
  By induction, we can prove that, for $m\geq 1$,
  \begin{align*}
    W^m=\sum_{i=0}^m W_i^{m-i}d^i,\ \
    W_i:=V+\delta(V)+\cdots+\delta^i(V)
  \end{align*}
  (see also \cite[Eq. (3.3.1) and (3.3.2)]{zhang1997note})
  and that,
  for $0\leq i\leq m$,
  \begin{align*}
   W_i=k\{1,
   z^{\pm1},z^jx,z^{-j}x^{-1},
   z^js: 0\leq j\leq i\}.
  \end{align*}
  Thus
  \begin{align*}
    W_i^{m-i}\subseteq k\{z^{p_1}x^{p_2}s^{p_3}:-i(m-i)\leq p_1\leq i(m-i),
    -(m-i)\leq p_2\leq m-i,
    0\leq p_3\leq m-i\}.
  \end{align*}
  As a result,
  \begin{align*}
    \dim(W^m)=
    \sum_{i=0}^m \dim(W_i^{m-i})
    \leq \sum_{i=0}^m(2i(m-i)+1)(2(m-i)+1)(m-i+1)\sim m^5,
  \end{align*}
  where $\sim$, meaning the equivalence of two functions in $\Phi$ (see Section \ref{sec_preliminaries}),
  follows from Lemma \ref{lemma-GK-sum}.
  Therefore, $\GK(A)={\limsup_{m\to\infty}}\log_m\dim (W^m)\leq 5$
  as desired.
\end{proof}

Proposition \ref{prop-Zhang97} together with Lemma \ref{lemma-Simplicity} give a criterion of the simplicity for certain GWAs.
\begin{corollary}
  \label{coro-Simple}
  Suppose that $k$ is algebraically closed and $A=D(\sigma,a)$ is a GWA,
  where $D$ is a \fg\ field and $a\neq 0$.
  Then the following statements are equivalent:
  \begin{enumerate}[(i)]
    \item
    $A$ is simple.
    \item
    $\GK(A)\neq\GK(D)+1$.
  \end{enumerate}
\end{corollary}
\begin{proof}
Note that if $D$ is a field and $a\neq 0$ then
conditions (i)--(iii) in Lemma \ref{lemma-Simplicity} are satisfied automatically.
As a result,
under the assumption,
$A$ is simple if and only if $\sigma^m$ is not inner for all $m\geq 1$,
if and only if $\sigma^m\neq \id$ for all $m\geq1$,
if and only if $\GK(A)\neq\GK(D)+1$
(by Proposition \ref{prop-Zhang97}).
\end{proof}

\

Now we are ready to prove Theorem \ref{thm-main-}.

\noindent
\emph{\textbf{Proof of Theorem \ref{thm-main-}.}}
It follows from Propositions \ref{prop-Zhang97PL} and \ref{prop-Zhang97}, and Corollary \ref{coro-Simple}.
\hfill\qed
\section{GK-dimension of GWAs over polynomial algebras}
\label{sce-Poly}
In this section,
we study GK-dimension of GWAs over $P_n:=k[z_1,z_2,\dots,z_n]$,
in particular over $P_2$.
We first prove a lemma, which will be used later.
\begin{lemma}
  \label{lemma_Power}
  Let $A=D(x,y;\sigma,a)$ be a GWA, $m\in\NNN^*$
  and $B$ be the subalgebra of $A$ generated by $D\cup \{x^m,y^m\}$.
  Then $B=D(x^m,y^m;\sigma^m,b)$ is a GWA and $\GK(A)=\GK(B)$, where
  \[
  b=b_m=\sigma^{-(m-1)}(a) \sigma^{-(m-2)}(a)\cdots \sigma^{-1}(a)a.
  \]
\end{lemma}
\begin{proof}
  First note that $x^md=\sigma^m(d)x^m$ and
   $y^md=\sigma^{-m}(d)y^m$ for all $d\in D$.
   Let \[
  b=\sigma^{-(m-1)}(a) \sigma^{-(m-2)}(a)\cdots \sigma^{-1}(a)a\in D.
  \]
  By using the defining relations of $A$,
  it is straightforward to check that $b=y^mx^m$.
  We claim that $b$ is central in $D$.
  For each $d\in D$ and $i\in\ZZ$,
  \[
  \sigma^i(a)d
  %=\sigma^i(a)\sigma^i(\sigma^{-i}(d))
  =\sigma^i(a\sigma^{-i}(d))
  =\sigma^i(\sigma^{-i}(d)a)
  %=\sigma^i(\sigma^{-i}(d))\sigma^i(a)
  =d\sigma^i(a),
  \]
  which implies $bd=db$ and the claim holds.

  Similarly, $\sigma^i(b)$ is central in $D$ for all $i\in \NNN$.
  Thus
  \begin{align*}
    \sigma^m(b)
    %&=\sigma^m(\sigma^{-(m-1)}(a) \sigma^{-(m-2)}(a)\cdots \sigma^{-1}(a)a)\\
    =\sigma(a)\sigma^2(a)\cdots\sigma^m(a)
    =\sigma^m(a)\sigma^{m-1}(a)\cdots\sigma(a)
    =x^my^m.
  \end{align*}
  Therefore, $B=D(x^m,y^m;\sigma^m,b)$ is a GWA.

  Since $A$ is \fg\ as a right $B$-module,
  we obtain by \cite[Proposition 5.5]{Krause-Lenagan_2000} that
  $\GK(A)=\GK(B)$.
\end{proof}

Now we can prove the following proposition.

\begin{proposition}\label{prop_GK_FiniteOrd}
  Suppose $A=D(\sigma,a)$ is a GWA.
  \begin{enumerate}[(i)]
    \item
    If $\sigma^m$ is an inner automorphism for some $m\geq 1$,
  then $\GK(A)=\GK(D)+1$.
  \item
  If $\sigma$ has a finite order then $\GK(A)=\GK(D)+1$.
  \end{enumerate}
\end{proposition}
\begin{proof}
(i) Let $A=D(x,y;\sigma,a)$ and $B$ be the subalgebra of $A$ generated by $D\cup \{x^m,y^m\}$.
Suppose $u\in D$ is an invertible element such that
  $\sigma^m(d)=udu^{-1}$ for all $d\in D$.
By Lemma \ref{lemma_Power}, $B=D(x^m,y^m;\sigma^m,b)$
where $b$ is a central element in $D$.
We claim further that $B=D(u^{-1}x^m,uy^m;\id,b)$.
For each $d\in D$,
\[
u^{-1}x^md=u^{-1}\sigma^m(d)x^m=u^{-1}udu^{-1}x^m=du^{-1}x^m
\]
and
$uy^md=duy^m$ similarly.
Note that $\sigma^{\pm m}(u)=u$ and $\sigma^{\pm m}(u^{-1})=u^{-1}$.
Thus \[
uy^mu^{-1}x^m=u\sigma^{-m}(u^{-1})y^mx^m=y^mx^m=b
\]
and
\[
u^{-1}x^m uy^m=x^m y^m=\sigma^m(b)=ubu^{-1}=b,
\]
where the last equality holds since
$b$ is central in $D$.
Thus,
$B=D(u^{-1}x^m,uy^m;\id,b)$.

Now it follows from Lemmas \ref{lemma_Power} and \ref{lem_GK-CA} that
$\GK(A)=\GK(B)=\GK(D)+1$.

(ii) follows from (i).
\end{proof}

Proposition \ref{prop-Zhang97} implies that the converse of Proposition \ref{prop_GK_FiniteOrd} (ii) also holds under certain conditions, while it does not hold in general,
see Example \ref{exam_FiniteOrd}.

\begin{corollary}
  Let $D$ be an algebra and $D[x,y]$ be the polynomial algebra in two variables over $D$.
  Suppose that
  $I$ is the ideal of $D[x,y]$ generated by $xy-a$ for some $a\in \Z(D)$.
  Then $\GK(D[x,y]/I)=\GK(D)+1$.
\end{corollary}
\begin{proof}
  By Proposition \ref{prop-universal},
  there exists a homomorphism $\psi:D(x,y;\id,a)\to D[x,y]/I$ such that $\psi(x)=x$, $\psi(y)=y$ and $\psi(d)=d$ for all $d\in D$.
  Note that the homomorphism
  \[
  \psi':D[x,y]\to D(x,y;\id,a), x\mapsto x,y\mapsto y, d\mapsto d, \text{ for all } d\in D
  \]
  maps $I$ to $0$ since
  $\psi'(xy-a)=xy-a=0$.
  Hence $\psi'$ induces a homomorphism $\psi'':D[x,y]/I\to D(x,y;\id,a)$.
  It is easy to check that $\psi''\psi=\id_{D(\id,a)}$ and $\psi\psi''=\id_{D[x,y]}$.
  Thus $D[x,y]/I\cong D(x,y;\id,a)$.
  It  follows from
  Proposition \ref{prop_GK_FiniteOrd} that
  $\GK(D[x,y])=\GK(D(x,y;\id,a))=\GK(D)+1$.
\end{proof}

Now we consider GWAs $P_n(\sigma,a)$ over $P_n$.
Denote by $\deg(f)$ the total degree of $f\in P_n$.
If $n=0$, then $\sigma=\id_k$ and $\GK(k(\sigma,a))=1$.
If $n=1$, then $\sigma\in\Aut(k[z])$ is determined by $\sigma(z)=bz+c$ for some $b\in k^*$ and $c\in k$.
Hence every $\sigma\in\Aut(k[z])$ is locally algebraic and thus $\GK(k[z](\sigma,a))=\GK(k[z])+1=2$.
For $n\geq2$, we have the following

\begin{lemma}
  \label{lem_GK-infty}
  Let $D=P_n$, $n\geq 2$, and $A=D(x,y;\sigma,a)$ be a GWA over $D$.
  Suppose that there exist $z\in D$ and real number $r\geq2$ such that
  \(
  \deg(\sigma^{m+1}(z))\geq r \deg (\sigma^{m}(z))
  \)
  for all $m\in\NNN$.
  Then $A$ has exponential growth.
  In particular, $\GK(A)=\infty$.
\end{lemma}
\begin{proof}
   Suppose there exist $z\in D$ and $r\geq2$ such that \(
  \deg( \sigma^{m+1}(z))\geq r \deg (\sigma^{m}(z))
  \)
  for all $m\in\NNN$.
  Consider elements of $A$ of the form
  \[
  z^{\epsilon_0}xz^{\epsilon_1}x\cdots z^{\epsilon_{p-1}}xz^{\epsilon_p},
  \]
  where $p\in\NNN$ and each $\epsilon_i$ is either $0$ or $1$.
  There are $2^{p+1}$ such elements (denoted by $a_1,\ldots,a_{2^{p+1}})$ for each $p\in\NNN$.
  It is easy to see that $a_i\in V^{2p+1}$ where $V=k+kz_1+\cdots+ kz_n+kx$.

  \emph{Claim.} $a_1,\ldots,a_{2^{p+1}}$ are linearly independent for all $p\in \NNN$.

  \emph{Proof of the claim.}
  Suppose there exists %(infinitely many)
  $q\in\NNN$ such that $a_1,\ldots,a_{2^{q+1}}$ are linearly dependent.
  Note that
  \begin{align*}
    z^{\epsilon_0}xz^{\epsilon_1}x\cdots z^{\epsilon_{p-1}}xz^{\epsilon_q}
    &=z^{\epsilon_0}\sigma(z^{\epsilon_1})\sigma^2(z^{\epsilon_2})\cdots \sigma^q(z^{\epsilon_q}) x^q\\
    &=z^{\epsilon_0}(\sigma(z))^{\epsilon_1}(\sigma^2(z))^{\epsilon_2}\cdots (\sigma^q(z))^{\epsilon_q} x^q
\end{align*}
  and thus the degree with respect to $x$ of each $a_i$ is $q$.
  Write $a_i=t_ix^q$, $i=1,2,\ldots,2^{q+1}$, $t_i\in D$.
  Since $a_1,\ldots,a_{2^{q+1}}$ are linearly dependent,
  so are $t_1,\ldots,t_{2^{q+1}}$.
  Thus there exist $t_i$ and $t_j$ such that $\deg (t_i)=\deg (t_j)$.
  Namely,
  there are two sequences $(\epsilon_0,\ldots, \epsilon_q)\neq (\epsilon'_0,\ldots, \epsilon'_q)$, each $\epsilon_i,\epsilon'_i\in\{0,1\}$,
  such that
  \begin{align}\label{eqn-deg}
  \epsilon_0d_0+\dots+\epsilon_qd_q
  =\epsilon'_0d_0+\dots+\epsilon'_qd_q
  \end{align}
  where each $d_i:=\deg( \sigma^i(z))$.
  Let $l$ be the maximal index such that $\epsilon_l\neq \epsilon'_l$.
  Clearly $l\geq1$.
  Without loss of generality, assume $\epsilon_l=1$ and $\epsilon'_l=0$.
  Then Eq. (\ref{eqn-deg}) becomes
  \[
  \epsilon_0d_0+\dots+\epsilon_{l-1}d_{l-1}+d_l
  %=\epsilon'_0d_0+\dots+\epsilon'_ld_l
  =\epsilon'_0d_0+\dots+\epsilon'_{l-1}d_{l-1}.
  \]
  By hypothesis,
  $d_{m+1}\geq rd_m$ for all $m\in\NNN$,
  which implies that $d_i\leq (1/r)^{l-i}d_l$.
  Thus
  \begin{align*}
    d_l&=(\epsilon'_0-\epsilon_0)d_0+\dots+
  (\epsilon'_{l-1}-\epsilon_{l-1})d_{l-1}\\
  &\leq   d_0+\dots+d_{l-1}
  \leq \sum_{i=0}^{l-1} \left(\frac1r\right)^{l-i}d_l
  =\frac {1-\left(\frac1r\right)^{l}}{r-1}d_l
  <d_l,
 \end{align*}
 a contradiction.
 Hence the claim holds.

 It follows from the claim that
 $\dim(V^{2p+1})\geq 2^{p+1}$ for all $p\in \NNN$ and thus $A$ has exponential growth.
\end{proof}

\begin{convention}\upshape
  \label{conv-auto}
For an automorphism $\sigma:P_n\to P_n$ defined by $\sigma(z_i)=f_i\in P_n$, $1\leq i\leq n$,
we denote
$\sigma=(f_1, f_2, \dots,f_n)$.
\end{convention}
Recall that
an automorphism $\sigma=(f_1, f_2, \dots,f_n)\in\Aut(P_n)$ is said to be \emph{triangular}
if
$
f_i=\lambda_i z_i+g_i,
$
where $\lambda_i\in k^*$ for $1\leq i\leq n$, $g_i\in k[z_{i+1},\ldots,z_n]$ for $1\leq i<n$ and $g_n\in k$.

\begin{lemma}
  \label{lemma_ConLAlgebraic}
  Suppose that $\sigma\in\Aut(P_n)$ is conjugate to a triangular automorphism.
  Then $\sigma$ is locally algebraic.
\end{lemma}
\begin{proof}
  Let $\delta, \tau\in \Aut(P_n)$ and $\tau=\delta^{-1}\sigma\delta$ be triangular.
Note that for every \fd\ subspace $U\subseteq P_n$,
the subset $\delta^{-1}(U)\subseteq P_n$ is also  \fd.
Since $\tau$ is triangular and
a triangular automorphism of $P_n$ is locally algebraic (\cite[Proposition 2]{Leroy1988On}),
we have that $\tau$ is locally algebraic.
Hence, there exists a \fd\ subspace $V_1\subseteq P_n$ such that
$\delta^{-1}(U)\subseteq V_1$ and $\tau(V_1)=V_1$.
Thus $\delta(V_1)$ is a \fd\ subspace of $P_n$ and $U\subseteq \delta(V_1)$.
Note that $\sigma(\delta(V_1))=\delta\tau\delta^{-1}(\delta(V_1))=\delta(V_1)$,
i.e.,
$\delta(V_1)$ is $\sigma$-stable.
Therefore $\sigma$ is locally algebraic.
\end{proof}

The converse of Lemma \ref{lemma_ConLAlgebraic} is true in the case $n=2$ (see Theorem \ref{thm_GKdichotomy}).
However it does not hold in general, see the following example.

\begin{example}
  \label{exam-Nagata}\upshape
  The Nagata automorphism (\cite{nagata1972automorphism})
  \[
  \left(
  x-2y(xz+y^2)-z(xz+y^2)^2, y+z(xz+y^2), z
  \right)
   \in \Aut(k[x,y,z])\]
is locally algebraic but not conjugate to a triangular automorphism (see \cite{bass1984non,furter2009quasi}).
\end{example}

\begin{corollary}
  \label{coro_PowerConjTri}
  Let $A=P_n(\sigma,a)$ be a GWA.
  If $\sigma^m$ is conjugate to a triangular automorphism for some $m\geq 1$,
  then $\GK(A)=n+1$.
\end{corollary}
\begin{proof}
  Let $D=P_n$, and $B$ and $b$ be as in Lemma \ref{lemma_Power}.
  Then $B=D(x^m,y^m;\sigma^m,b)$ and $\GK(A)=\GK(B)$ by Lemma \ref{lemma_Power}.
  By the assumption and Lemma \ref{lemma_ConLAlgebraic}, $\sigma^m$ is locally algebraic and thus $\GK(B)=n+1$
  by Lemma \ref{lem_GK-CA}.
  Hence $\GK(A)=n+1$.
\end{proof}

Now we consider the case $n=2$.
\begin{theorem}
  \label{thm_GKdichotomy}
  Let $A=k[z_1,z_2](\sigma,a)$ be a GWA.
  Then $\GK(A)=3$ if $\sigma$ is conjugate to a triangular automorphism, and $\GK(A)=\infty$ otherwise.
\end{theorem}
\begin{proof}
Let $D=k[z_1,z_2]$.
Suppose that $\sigma$ is conjugate to a triangular automorphism.
By Corollary \ref{coro_PowerConjTri},  $\GK(A)=3$.

Suppose that $\sigma$ is not conjugate to a triangular automorphism.
Let $\pi=(z_2,z_1)\in\Aut(D)$ (see Convention \ref{conv-auto}).
 By \cite[Lemma 3]{Lane1975Fixed},
 $\sigma$ is conjugate to an automorphism $\sigma'$ of the form
 \begin{align*}
   \label{eqn-square}
 \sigma'=\tau_1\pi\tau_2\cdots\pi\tau_s\pi,\
 s\geq 1,
 \end{align*}
 where,
 for each positive integer $i\leq s$,
 \[
 \tau_i=(\alpha_i z_1+\beta_i(z_2),u_iz_2+v_i),\
 \alpha_i,u_i\in k^*, v_i\in k,\
 \beta_i(z_2)\in k[z_2],
 d_i:=\deg(\beta_i)>1.
  \]
 We claim that
 \[
 \sigma'(z_2)=\lambda z_2^{d_1\cdots d_s}+\text{(lower total degree terms)}, \ \lambda\in k^*.
 \]
 We prove the claim by induction on $s$.
 If $s=1$,
 then $\sigma'(z_2)=\tau_1(z_2)=\alpha_1z_1+\beta_1(z_2)
 =\lambda z_2^{d_1}+\alpha_1z_1$.
 Thus the claim holds for $s=1$.
 Suppose the claim holds for automorphisms with $s\geq 1$.
 For $\sigma'=\tau_1\pi\tau_2\cdots\pi\tau_{s+1}\pi$,
 we have that
 \begin{align*}
   \sigma'(z_2)&=\tau_1\pi
   \left(\lambda z_2^{d_2d_3\cdots d_{s+1}}+T \right)\\
   &=\tau_1
   \left(\lambda z_1^{d_2d_3\cdots d_{s+1}} \right)+\tau_1\pi(T)\\
   &=\lambda (\alpha_1z_1+\beta_1(z_2))^{d_2d_3\cdots d_{s+1}}+\tau_1\pi(T)\\
   &=\lambda' z_2^{d_1d_2\cdots d_{s+1}}+T' +\tau_1\pi(T)
 \end{align*}
 where $\lambda,\lambda'\in k^*$,
$\deg (T)<d_2d_3\cdots d_{s+1}$
  and $\deg (T')<d_1d_2\cdots d_{s+1}$.
 Since $\deg (\tau_1\pi(T))=\deg (\tau_1(T))<d_1d_2\cdots d_{s+1}$,
 we have that
 \[
 \sigma'(z_2)=\lambda' z_2^{d_1d_2\cdots d_{s+1}}+\text{ (lower total degree terms)}
  \]
  as desired.

 It follows from the claim that
\begin{align*}
 \deg (\sigma'^{m+1}(z_2))&= \deg((\tau_1\pi\tau_2\cdots\pi\tau_s\pi)^{m+1}(z_2))\\
 &=(d_1\cdots d_s)^{m+1}\\
 &\geq (d_1\cdots d_s)^{m}=2\deg (\sigma'^{m}(z_2)).
\end{align*}
Now, by Lemma \ref{lem_GK-infty},
$D(\sigma',b)$ has exponential growth and
$\GK(D(\sigma',b))=\infty$ for all $b\in k[z_1,z_2]$.
Since $\sigma$ is conjugate to $\sigma'$,
we may suppose $\sigma=\delta^{-1}\sigma'\delta$ for some $\delta\in \Aut(k[z_1,z_2])$.
Then, by Lemma \ref{lemma_iso}, we have $k[z_1,z_2](\sigma,a)\cong k[z_1,z_2](\sigma',\delta^{-1}(a))$.
Hence $\GK(D(\sigma,a))=\GK(D(\sigma',\delta^{-1}(a)))=\infty$.
\end{proof}

\begin{remark}\upshape
  \label{remark-P2-P3}
Note that a key element in
the proof of Theorem \ref{thm_GKdichotomy} is \cite[Lemma 3]{Lane1975Fixed},
which is based on the famous theorem of Jung and van der Kulk \cite{Jung1942Uber,vanderKulk1953} that asserts
that $\Aut(P_2)$ is generated by triangular automorphisms and
the permutation $\pi=(z_2,z_1)$,
that is,
all $\sigma\in\Aut(P_2)$ can be written as a product of triangular automorphisms (and their inverses) and $\pi$.
Shestakov and Umirbaev \cite{Shestakov2003}
proved that there exist automorphisms $\tau\in\Aut(P_3)$ that are not tame.
As a result,
such $\tau$ cannot be presented as a product of triangular automorphisms (and their inverses) and affine automorphisms (i.e., automorphisms $\delta=(f_1,f_2, f_3)\in\Aut(P_3)$ with $\deg(f_i)=1$ for $i=1,2,3$) of $P_3$.
\end{remark}

\

We can now put everything together to prove Theorem \ref{thm-mainP}.

\noindent
\emph{\textbf{Proof of Theorem \ref{thm-mainP}.}}
(\ref{itemP-PolyGrowth}) $\Rightarrow$
(\ref{itemP-NotExGrowth})
and (\ref{itemP-ConjTri}) $\Rightarrow$
(\ref{itemP-PowerConjTri})
are trivial.
(\ref{itemP+1}) $\Leftrightarrow$
(\ref{itemP-GKfinite}) $\Leftrightarrow$ (\ref{itemP-LA})
%$\Leftrightarrow$ (\ref{itemP-ConjTri})
follow from Theorem \ref{thm_GKdichotomy} and
Proposition \ref{prop-Zhang97PL} respectively.
(\ref{itemP-NotExGrowth}) $\Rightarrow$ (\ref{itemP-ConjTri})
follows from the proof of Theorem \ref{thm_GKdichotomy}.
(\ref{itemP-PowerConjTri}) $\Rightarrow$ (\ref{itemP+1})
follows from Corollary \ref{coro_PowerConjTri}.

(\ref{itemP-LA}) $\Rightarrow$ (\ref{itemP-PolyGrowth}).
Suppose $\sigma$ is locally algebraic.
Let $V$ be a $\sigma$-stable generating subframe of $P_2$ such that $a\in V$ (and thus $\sigma(a)\in V$).
Then $W:=V+kx+ky$ is a generating subframe of $A$.
Since
  \[
  W^{2m}=(V+kx+ky)^{2m}\supseteq V^m+V^mx+\cdots+V^mx^m, m\geq1,
  \]
we have that
\begin{align}
  \label{eqn-Lower}
  \dim(W^{2m})> (m+1)\dim(V^m), m\geq1.
\end{align}
On the other hand,
we claim that
  \begin{align}\label{claim_W^n}
    W^m\subseteq\sum_{i=1}^mV^{m}x^i+\sum_{i=1}^mV^{m}y^i+V^{m}, \ m\geq 1.
  \end{align}
The claim is obvious when $m=1$.
Suppose that the inclusion (\ref{claim_W^n}) holds for some $m\geq 1$. Then
 \begin{align*}
   W^{m+1}&\subseteq W\left(\sum_{i=1}^mV^{m}x^i+\sum_{i=1}^mV^{m}y^i+V^{m}\right)\ \ (\text{induction hypothesis})\\
   &=kx\sum_{i=1}^mV^{m}x^i+kx\sum_{i=1}^mV^{m}y^i+kxV^{m}\\
   &\ \ \ +ky\sum_{i=1}^mV^{m}x^i+ky\sum_{i=1}^mV^{m}y^i+kyV^{m}\\
 &\ \ \ \ \ \ +V\sum_{i=1}^mV^{m}x^i+V\sum_{i=1}^mV^{m}y^i+VV^{m}\\
 &=\sum_{i=1}^mV^{m}x^{i+1}+ \sum_{i=1}^mV^{m}\sigma(a)y^{i-1}+V^{m}x\ \ [\text{since } xV^{m}=V^{m}x]\\
   &\ \ \ +\sum_{i=1}^mV^{m}ax^{i-1}+\sum_{i=1}^mV^{m}y^{i+1}+V^{m}y\ \ [\text{since } yV^{m}=V^{m}y]\\
 &\ \ \ \ \ \ +\sum_{i=1}^mV^{m+1}x^i+\sum_{i=1}^mV^{m+1}y^i+V^{m+1}\\
 &\subseteq\sum_{i=1}^{m+1}V^{m+1}x^i+\sum_{i=1}^{m+1}V^{m+1 }y^i+V^{m+1 }\ \ [\text{since } a,\sigma(a)\in V].
  \end{align*}
Hence the inclusion (\ref{claim_W^n}) holds for  $m+1$ and thus
for all $m\geq 1$.

By the $\ZZ$-graded structure of GWAs, the sum in Eq. (\ref{claim_W^n}) is direct and thus
\begin{align}
  \label{eqn-Upper}
  \dim(W^m)\leq (2m+1)\dim(V^{m}), m\geq1.
\end{align}
Combining Eq. (\ref{eqn-Lower}) and (\ref{eqn-Upper}) gives
\[
(m+1)\dim(V^m)\leq \dim(W^{2m})\leq (4m+1)\dim(V^{2m}), m\geq 1.
\]
Since the growth of $P_2$ is independent from the choice of generating subspaces,
$\dim(V^m)\sim \dim(kz_1+kz_2)^m\sim m^2$.
We obtain that $\dim(W^m)\sim m^3$ and thus $A$ has polynomial growth.
\begin{figure}[h!]
\begin{align*}
  \xymatrix{
  (a) \ar@2{<-}[d]_{\ref{coro_PowerConjTri}} \ar@2{<->}[r]^{\ref{thm_GKdichotomy}}
                & (b) \ar@2{<->}[r]^{\ref{prop-Zhang97PL}}
                &  (e)\ar@2{->}[r]
                & (c)  \ar@2{->}[d] \\
  (g)  \ar@2{<-}[r]
                & (f)
                &
                &(d)  \ar@2{->}[ll]_{\text{Proof of } \ref{thm_GKdichotomy}}   }
\end{align*}
\caption{Implication relations among the items in Theorem \ref{thm-mainP}}
\label{fig}
\end{figure}
Figure \ref{fig} summarizes the relations we have proved among the items in Theorem \ref{thm-mainP},
which completes the proof.
\hfill \qed

\

We conclude this section with an example,
which shows that the dichotomy of GK-dimension appears in Theorem \ref{thm_GKdichotomy} (i.e., $\GK(P_2(\sigma,a))$ is either $3$ or $\infty$)
cannot be extended to that of GWAs over a general base algebra of GK-dimension two.
Note that
it is proved by Rogalski \cite{rogalski2009gk} that
if $k$ is algebraically closed and $K/k$ is a finitely generated field extension with
$\GK(K)=2$,
then every ``big subalgebra" \cite[Definition 6.1]{rogalski2009gk} of the GWA $K(x,y;\sigma,1)=K[x,x^{-1};\sigma]$
has the same GK-dimension $d\in\{3,4,5,\infty\}$.
\begin{example}
  \label{exam_Heisenberg}\upshape
  Let
  \[
  H_m=H_m(\alpha_1,\alpha_2):=k[z_1^{\pm1},z_2^{\pm1}](x,y;\sigma_m,a)
  \]
  be a GWA, where
  \[
  \sigma_m(z_1)=\alpha_1z_1, \sigma_m(z_2)=\alpha_2z_1^mz_2, \
  \alpha_1,\alpha_2\in k^*, m\in\NNN.
  \]
  Then $\GK(H_m)=4$ for all $m\neq 0$ and $\GK(H_0)=3$.
\end{example}
\begin{proof}
Without loss of generality,
we assume $\alpha_1=\alpha_2=1$.
In fact,  since $V=k\{z_1^{\pm1},z_2^{\pm1}, x,y\}$ is a generating subspace of both $H_m(\alpha_1,\alpha_2)$ and $H_m(1,1)$,
we have that $\GK(H_m(\alpha_1,\alpha_2))=\GK(H_m(1,1))$.
We may also suppose $m\geq 0$, since
$H_m\cong H_{-m}$ via the isomorphism determined by
$z_1\mapsto z_1^{-1}$, $z_2\mapsto z_2$, $x\mapsto x$, $y\mapsto y$.

It is clear that $\sigma_0$ is locally algebraic and thus
$\GK(H_0)=3$. Now assume $m\geq1$.
Note that $H_1$ is the group algebra of the \emph{first (discrete) Heisenberg group}
(also known as the \emph{generic quantum torus} \cite[Example 2.5(ii)]{jordan2009reversible}).
It is well-known that $\GK(H_1)=4$ (see e.g. \cite[Example 11.10]{Krause-Lenagan_2000}).
 By Lemma \ref{lemma_Power},
  the subalgebra $B=k[z_1^{\pm1},z_2^{\pm1}](x^m,y^m;\sigma_1^m,1)\subseteq H_1$ is a GWA and $\GK(B)=\GK(H_1)=4$.
  By the universal property of GWAs (Proposition \ref{prop-universal}), the mappings
  \[
  \psi:H_m\to B, z_1\mapsto z_1, z_2\mapsto z_2, x\mapsto x^m, y\mapsto y^m
  \]
  and
  \[
  \psi':B\to H_m, z_1\mapsto z_1, z_2\mapsto z_2, x^m\mapsto x, y^m\mapsto y
  \]
  determine two algebra homomorphisms such that
  $\psi\psi'=\id_{B}$ and $\psi'\psi=\id_{H_m}$.
  Hence $H_m\cong B$ and thus $\GK(H_m)=\GK(B)=4$.
\end{proof}

%%%%%%%%%%%%%%%%%%%%%%%%%%%%
\section{GK-dimension of GWAs over Laurent polynomial algebras}
\label{sce-Laurent}
In this section, we first characterize the automorphisms of Laurent polynomial algebra
$$
L_n=k[z_1^{\pm1},z_2^{\pm1},\dots,z_n^{\pm1}]
$$
 and then investigate the GK-dimension of GWAs over $L_n$.
Finally we prove Theorem \ref{thm-mainL} and give an application on cancellation problem.

%\subsection{Automorphisms of  Laurent polynomial algebras}
Denote by $\ZZ^{m\times n}$ (respectively $\GL(n,\ZZ)$) the set of all $m\times n$ integer matrices (respectively the general linear group of degree $n$ over $\ZZ$).
%Let $\vv0_n$ and $I_n$ (or simply $\vv0$ and $I$ respectively) be the zero matrix and the identity matrix in $\ZZ^{n\times n}$ respectively.
%Note that $\det M=\pm1$ for all $M\in\GL(n,\ZZ)$.
For $M=(a_{ij})\in\ZZ^{n\times n}$, let $M[i,j]:=a_{ij}$ be the $(i,j)$-th element of $M$,
$M[i,-]:=(a_{i1},a_{i2},\dots,a_{in})$ be the $i$-th row of $M$,
and $M[-,j]:=(a_{1j},a_{2j},\dots,a_{nj})^T$ be the $j$-th {column} of $M$,
 where $T$ stands for the transpose operation.

{It is well-known} that the unit group of $L_n$ is
\[
U(L_n)=\{
\alpha z_1^{i_1}z_2^{i_2}\cdots z_n^{i_n}:\alpha\in k^*, i_1,i_2,\dots,i_n\in \ZZ\}.
\]
Since an automorphism $\sigma\in\Aut(L_n)$ sends a unit to a unit,
we have
\begin{align} \label{eqn_auto}
\sigma(z_i)=\alpha_iz_1^{a_{1i}}z_2^{a_{2i}}\cdots z_n^{a_{ni}},
 \ \alpha_i\in k^*, a_{1i},a_{2i},\dots, a_{ni}\in\ZZ,\
i=1,2,\dots,n.
\end{align}
The matrix
$M=M_{\sigma}=(a_{ij})\in \ZZ^{n\times n}$
(respectively, the scalar $\alpha=\alpha_{\sigma}:=(\alpha_1,\alpha_2,\dots,\alpha_n)\in (k^*)^n$)
is uniquely determined by the automorphism $\sigma$ and we write $\sigma=(M,\alpha)$.
%Call $\sigma$ a \emph{scalar automorphism} (respectively \emph{monomial automorphism}) if
%$M_{\sigma}=I$ (respectively $\alpha_{\sigma}=(1,1,\dots,1)$).
We introduce the following convention.
\begin{convention}\upshape
  \label{conv-a^b}
Denote $z:=(z_1,z_2,\dots,z_n)$.
Given rows $\alpha=(\alpha_1,\alpha_2,\dots,\alpha_n)\in L_n^n$ and
$\beta=(\beta_1,\beta_2,\dots,\beta_n)\in L_n^n$,  a column $b=(b_1,b_2,\dots,b_n)^T\in \ZZ^{n\times 1}$,
and a matrix $M\in \ZZ^{n\times n}$,
we use the following notation:
\begin{align*}
  \alpha^b&:=
\alpha_1^{b_1}\alpha_2^{b_2}\dots \alpha_n^{b_n}\in L_n,\\
\alpha^M&:=
(\alpha^{M[-,1]},\alpha^{M[-,2]},\dots,\alpha^{M[-,n]})\in L_n^n,\\
\alpha \beta&:=(\alpha_1\beta_1,\alpha_2\beta_2,\dots,\alpha_n\beta_n)\in L_n^n.
\end{align*}
\end{convention}
In particular, the element-wise multiplication
\[
\alpha\beta=(\alpha_1\beta_1,\alpha_2\beta_2,\dots,\alpha_n\beta_n),\
\forall \alpha,\beta\in (k^*)^n
\]
gives a group structure on $(k^*)^n$.

With the Convention \ref{conv-a^b}, we can write the automorphism $\sigma=(M,\alpha)$ given in Eq. (\ref{eqn_auto}) as
\begin{align}\label{eqn_auto2}
  \sigma(z_i)&=\alpha_iz_1^{M[1,i]}z_2^{ M[2,i]}\cdots z_n^{ M[n,i]}=
 \alpha_i z^{M[-,i]},\ \ i=1,2,\dots,n
\end{align}
or
\begin{align}
\label{eqn_sigma(z)}
  \sigma(z):=(\sigma(z_1),\sigma(z_2),\dots,\sigma(z_n))
  =(\alpha_1z^{M[-,1]},\alpha_2z^{M[-,2]},\dots,\alpha_nz^{M[-,n]})
  =\alpha z^M.
\end{align}

\begin{lemma}\label{lemma_exponent}
Suppose $\alpha,\beta,\gamma\in {L_n^n}$ and
$M,N\in \ZZ^{n\times n}$.
Then
\begin{enumerate}[(i)]
%  \item
%  $\alpha^I=\alpha$.
%  \item
%  $\vv1^{M}=\vv1=\alpha^{\vv0}$.
  \item
  $(\alpha\beta)^M=\alpha^M\beta^M$.
  \item
  $(\alpha^M)^N=\alpha^{MN}$.
  \item
  $\alpha^M\alpha^N=\alpha^{M+N}$.
%  \item
%  $(\alpha\beta)\gamma=\alpha(\beta\gamma)$.
\end{enumerate}
\end{lemma}
\begin{proof}
  (i) and (iii) are obvious.
  (ii) can be checked directly.
\end{proof}

\begin{lemma}
  \label{lemma_composition}
  Suppose $\sigma=(M,\alpha), \tau=(N,\beta)\in \Aut(L_n)$.
  Then
  \begin{enumerate}[(i)]
   \item
    $\tau\sigma=(N,\beta)(M,\alpha)=(NM,\beta^M\alpha).$
    \item
    $M$ is invertible and $\sigma^{-1}= (M^{-1},\alpha^{-M^{-1}})$.
  \end{enumerate}
\end{lemma}
\begin{proof}
(i) Suppose  $C=NM$.
By Eq. (\ref{eqn_auto2}), it is routine to check that, for $1\leq i\leq n$,
$(\tau\sigma)(z_i)=\beta^{M[-,i]}\alpha_iz^{C[-,i]}.$
Thus
$\tau\sigma(z)=\beta^M\alpha z^{NM}$ and
$\tau\sigma=(NM,\beta^M\alpha)$.

(ii) follows from (i) and Lemma \ref{lemma_exponent}.
\end{proof}

The following theorem is well known,
see for instance \cite[Lemma 3.6(3)]{CPWZ15} or \cite[Exercise 5.19 (b)]{BG09}.

\begin{theorem}
  \label{thm_auto(L)}
  Let $L_n=k[z_1^{\pm1},z_2^{\pm1},\dots,z_n^{\pm1}]$.
  Then \(
  \Aut(L_n)\cong \GL(n,\ZZ)\ltimes (k^*)^n
  \) as groups,
  where the group operation of the semiproduct is given by Lemma \ref{lemma_composition} (i).
\end{theorem}
\begin{proof}
It follows from Lemma \ref{lemma_composition} (ii) that the mapping
  \[
  f:\Aut(L_n)\to \GL(n,\ZZ)\times (k^*)^n,\ \sigma\mapsto(M_{\alpha},\alpha_{\sigma})
  \]
  is well-defined and bijective.
It is routine to check that
 $\GL(n,\ZZ)\times (k^*)^n$ has a semidirect structure with
the product given by Lemma \ref{lemma_composition} (i).
\end{proof}

%%%%%%%%%%%%%%
%\subsection{GK-dimension of GWAs over $L_n$}
\begin{definition}
Suppose $M\in\ZZ^{n\times n}$ and
\[
\|M\|:=\max_{1\leq i,j\leq n} |M[i,j]|.
\]
\begin{enumerate}[(i)]
  \item
  $M$ is called \emph{power-bounded} (by $N$)
if there exists $N\in\NNN^*$ such that $\|M^m\|<N$
for all $m\in \NNN$.
  \item
  $M$ has a \emph{finite order} if $M^m=I$ for some $m\in\NNN^*$, and the smallest such $m$ is called the \emph{order} of $M$, denoted by $\oo(M)$.
\end{enumerate}
\end{definition}

Note that if $\sigma=(\alpha,M)\in\Aut(L_n)$ has a finite order,
then so does $M$.
But, as demonstrated in Example \ref{exam_FiniteOrd},
the other way around is not true in general.

\begin{lemma}
  \label{lemma_pb=la}
  Let $\sigma=(\alpha,M)\in \Aut(L_n)$, $a\in L_n$ and $A=L_n(x,y;\sigma,a)$.
  Then the following statements are equivalent:
  \begin{enumerate}[(i)]
    \item
    $M$ is power-bounded.
    \item
    $M$ has a finite order.
    \item
    $\sigma$ is locally algebraic.
  \end{enumerate}
\end{lemma}
\begin{proof}
(i)$\Rightarrow$(ii).
Suppose that $\|M^m\|<N$ for some $N\geq 1$ and all $m\in\NNN$.
Then the entries of matrix $M^m$ must be integers between $-N$ and $N$ for all $m\in \NNN$.
Hence $\{M^m:m\in \NNN\}$ is a finite subgroup of $\GL(n,\ZZ)$.
Thus $M$ is of finite order.

(ii)$\Rightarrow$(iii).
Suppose $M$ has a finite order.
Then $\{M^i:i\geq 0\}$ is a finite set.
Let $N$ be the maximal number that appears as an entry of the matrices $\{M^i:i\geq 0\}$.

\emph{Claim.}  $\sigma^m(z)=\beta_m z^{M^m}$ for $m\geq 1$, where
\[
\beta_m=\alpha^{I+M+M^2+\dots+M^{m-1}}, m\geq 1.
\]

\emph{Proof of the claim.} The claim holds for $m=1$ by Eq. (\ref{eqn_sigma(z)}).
Assume that the claim holds for some $m\geq 1$.
Then, by Eq. (\ref{eqn_sigma(z)}) and Lemma \ref{lemma_exponent},
\begin{align*}
  \sigma^{m+1}(z)
  =\sigma(\beta_m z^{M^m})
%  &=\beta_m\sigma\left( z_1^{M^m[1,-]},z_2^{M^m[2,-]},\dots,z_n^{M^m[n,-]}\right)\\
=\beta_m\sigma( z)^{M^m}
=\beta_m( \alpha z^M)^{M^m}
=\beta_m \alpha^{M^m} z^{M^{m+1}}
=\beta_{m+1} z^{M^{m+1}}.
\end{align*}
Thus the claim holds for all $m\geq 1$.

By the claim, we have that
 \[
 \sigma^m(z_i)\in V:= k\{z_1^{i_1}z_2^{i_2}\dots z_n^{i_n}:-N\leq i_j\leq N, 1\leq j\leq n\},  m\geq 1,
 \]
where $V$ is a \fd\ subspace with $\dim(V)\leq (2N+1)^n$.
Suppose $T=z_1^{a_1}z_2^{a_2}\dots z_n^{a_n}\in L_n$, each $a_i\in \ZZ$.
Then
 $\sigma^m(T)=(\sigma(z_1))^{a_1}(\sigma(z_2))^{a_2}\dots (\sigma(z_n))^{a_n}$,
 which is contained in the \fd\ subspace
  \[
 V_T:=k\{z_1^{i_1}z_2^{i_2}\dots z_n^{i_n}:-N'\leq i_j\leq N', 1\leq j\leq n\}
 \]
 for all $m\geq 1$, where $N'=N(|a_1|+|a_2|+\dots+|a_n|)$.
 As a result,
 for each Laurent polynomial $p\in L_n$, say
 \[
  p=c_1T_1+c_2T_2+\dots+c_sT_s,
  c_i\in k^*, T_i \text{ Laurent monomial}, i=1,2,\dots,n,
  \]
 we have that $\sigma^m(p)$ is contained in the \fd\ space
 $V_{T_1}+V_{T_2}+\dots+V_{T_s}$
 for all $m\geq 1$.
 Therefore, by Lemma \ref{lemma-LocAlgIff},
 $\sigma$ is locally algebraic.

 (iii)$\Rightarrow$(i).
 Suppose that $M$ is not power-bounded,
 i.e., for every $l>1$, there exists $m=m(l)>1$ such that
 $\|M^m\|>l$.
 Without loss of generality,
 assume $\|M^m\|=|M^m[1,1]|$ for infinitely many $m\geq 1$.
 Let $V=kz_1+kz_2+\dots+kz_n+kz_1^{-1}+kz_2^{-1}+\dots+kz_n^{-1}$.
 It follows from the above claim that
 \[
 \sigma^m(z_1)=cz_1^{M^m[1,1]}z_2^{M^m[2,1]}\dots z_n^{M^m[n,1]}\not\in
 V^l.
 \]
Thus the subspace $k\{\sigma^m(z_1):m\geq 1\}$ is not \fd.
Therefore $\sigma$ is not locally algebraic.

\end{proof}

\noindent
\emph{\textbf{Proof of Theorem \ref{thm-mainL}.}}
  It follows from Proposition \ref{prop-Zhang97PL} and Lemma \ref{lemma_pb=la}.
  \hfill\qed

\begin{remark}\upshape
  \label{remark-MaxOrd}
It follows from Lemma \ref{lemma_pb=la} that
the elements of finite order in $\GL(n,\ZZ)$ play an important role.
Given $n\geq 1$, there are only finitely many integers that can be the orders of matrices in $\GL(n,\ZZ)$.
Levitt \cite{Levitt1998} shows that the maximum order $G(n)$ of elements of finite order in $\GL(n,\ZZ)$ satisfies a Landau-type estimate $\ln G(n)\sim \sqrt{n\ln n}$ and
there is a method of dynamical programming to compute $G(n)$.
Following \cite{Levitt1998}, Table \ref{tab-Gn} gives the values of $G(n)$ for $n\leq 20$ where $G(n)$ is omitted for odd $n=2p+1$ since $G(2p+1)=G(2p)$.
\begin{table}[h]
  \centering
  \begin{spacing}{1.15}
  \caption{Maximum orders of matrices in $\GL(n,\ZZ)$}\label{tab-Gn}
  \begin{tabular}{c|ccccccccccc}
  \hline
  $n$&$1$&$2$&$4$&$6$&$8$&$10$
  &$12$&$14$&$16$&$18$&$20$\\
  \hline
  $G(n)$&$2$&$6$&$12$&$30$&$60$&$120$
  &$210$&$420$&$840$&$1260$&$2520$\\
  \hline
\end{tabular}
\end{spacing}
\end{table}
By using \cite[Proposition 1.1]{Levitt1998},
we can determine whether an integer is the order  of an element of $\GL(n,\ZZ)$.
Note that for a given order $o$ there may be many matrices of order $o$ (see Example \ref{exam_FiniteOrd}).
\end{remark}

\begin{example}
\label{exam_FiniteOrd}\upshape
  Suppose that $c\in k^*$ and $\sigma=\sigma_q \in \Aut(L_2)$ is defined as
  \[
  \sigma(z_1)=cz_1^{1-q}z_2^{2-q},
  \sigma(z_2)=z_1^{q}z_2^{q-1},\ q\in\ZZ.
  \]
  Then $\sigma^m\neq \id$ for all $m\geq 1$.
But $M_{\sigma}=\begin{pmatrix}
    1-q&q\\
    2-q&q-1
  \end{pmatrix}$ has a finite order $\oo(M)=2$ for all $q\in\ZZ$
and thus $\GK(L_2(\sigma,a))=\GK(L_2)+1=3$
by Theorem \ref{thm-mainL}.
\end{example}

We conclude this section by an application of Theorem \ref{thm-mainL} to cancellation problem.
\begin{definition}[Definitions 0.1 and 0.8 of \cite{tang2020cancellation}]
Let $A$ be an algebra.
\begin{enumerate}[(i)]
\item
We say $A$ is \emph{cancellative} (respectively, \emph{strongly cancellative})
if the isomorphism
\begin{align*}
 A[x]\cong B[x] \ \ (\text{respectively }
 A[x_{1}, \cdots, x_{m}] \cong B[x_{1}, \cdots, x_{m}])
\end{align*}
implies $A\cong B$
for all algebra $B$ (respectively for all algebra $B$ and all integer $m\geq 1$).
\item
We say $A$ is \emph{$\sigma$-algebraically cancellative} (respectively, \emph{strongly $\sigma$-algebraically cancellative})
if the isomorphism between Ore extensions (respectively iterated Ore extensions)
\begin{align*}
 A[x;\sigma']\cong B[x;\sigma''] \ \ (\text{respectively }
 A[x_{1};\sigma_1']\cdots[x_{m};\sigma_m'] \cong B[x_{1};\sigma_1'']\cdots[x_{m};\sigma_m''])
\end{align*}
implies $A\cong B$
for all algebra $B$ and locally algebraic automorphisms $\sigma'$ and $\sigma''$
(respectively for all algebra $B$,
 locally algebraic automorphisms $\sigma'_i$ and $\sigma''_i$, $1\leq i\leq m$, and all integer $m\geq 1$).
\end{enumerate}
\end{definition}
The following lemma follows directly from the definitions and its proof is omitted.

\begin{lemma}
  \label{lemma-SCancel}
  Let $A$ be an algebra.
  \begin{enumerate}[(i)]
    \item
    If $A$ is strongly cancellative (respectively strongly $\sigma$-algebraically cancellative),
  then $A$ is cancellative (respectively $\sigma$-algebraically cancellative).
    \item
    If $A$ is $\sigma$-algebraically cancellative (respectively strongly $\sigma$-algebraically cancellative),
  then $A$ is cancellative (respectively strongly cancellative).
  \end{enumerate}

\end{lemma}

Theorem \ref{thm-mainL} together with results from \cite{tang2020cancellation} can be applied to get some cancellation properties of GWAs.

\begin{corollary}
  \label{coro-Cancellation}
  Let $A=L_n(\tau,a)$ be a GWA with $a\neq 0$ such that one of the equivalent conditions in Theorem \ref{thm-mainL} is satisfied.
  Then $A$ is strongly $\sigma$-algebraically cancellative.
  As a consequence,
  $A$ is strongly cancellative and cancellative.
\end{corollary}
\begin{proof}
  Since $L_n$ is a \fg\ domain and $a\neq 0$,
  it follows from Lemma \ref{lemma-Noeth} that $A$ is a \fg\ domain.
  It is easy to check that the identity element $1$ is a \emph{controlling element} (see \cite[Definition 4.2]{tang2020cancellation}) of $A$, i.e., ${\mathbb D}(1)=A$.
  By the assumption and Theorem \ref{thm-mainL},
  $\GK(A)=n+1<\infty$ and thus
  $A$ is strongly $\sigma$-algebraically cancellative by \cite[Theorem 0.9]{tang2020cancellation}.
\end{proof}
%%%%%%%%

\

\textbf{Acknowledgements.}
The author would like to thank the anonymous referees for their
careful reading, valuable suggestions and useful references.
Part of this work was done during a visit to James J. Zhang at the University of Washington.
The author would like to express deep gratitude to James J. Zhang
for many stimulating conversations on the subject,
and to the department of Mathematics at the University of Washington
for the hospitality during his visit.
The author wishes to thank Yu Li, Wenchao Zhang and Xiaotong Sun for valuable discussions.
This work was supported in part by Huizhou University (hzu202001, 2021JB022) and the Guangdong Provincial Department of Education (2020KTSCX145, 2021ZDJS080).

\

\section*{Appendix: proof of Lemma \ref{lemma-PnSMC}}
\label{Sec_appendix}
\begin{proof}
For the reader's convenience, we reproduce a proof by copying that of \cite[Theorem 3.2]{zhang1997note} with small modification.
  We use induction on $n$.
  When $n=0$, $P_0=k$, it is easy to check that $k$ satisfies $\SM(k,1,0)$.
  Suppose $P_n$ satisfies $\SM(V_n,c_n,n)$.
  We need to prove that $P_{n+1}$ satisfies $\SM(V,c,n+1)$
  for $V=V_{n+1}$ and $c=c_{n+1}>0$.
  Denote $z=z_{n+1}$.
  Suppose $W\subseteq P_{n+1}$ is \fd\ and
  $Va\subseteq W$ for some $0\neq a\in P_{n+1}$.
  It suffices to show that
  \begin{align}
    \label{eqn-W^m}
    \dim(W^m)\geq c\dim(W) m^{n+1}, m\in\NNN.
  \end{align}
  Case 1. Suppose $W$ has the form
    \[
    W=W_1z^{p_1}+W_2z^{p_2}+\cdots+W_lz^{p_l},\ \ l\geq 2, p_1<p_2<\dots<p_l
    \]
    such that
    $V_na_i\subseteq W_i\subseteq P_n$ for $0\neq a_i\in P_n$.
    In this case, we prove by induction on $l$ that
    $W$ satisfies the following inequality and thus satisfies Eq. (\ref{eqn-W^m}):
  \begin{align}
    \label{eqn-W^mCase1}
    \dim(W^m)\geq c'\dim(W) m^{n+1}, \ c'=\frac{c_n}{2(n+1)}, m\in\NNN.
  \end{align}
    If $l=2$, i.e., $W=W_1z^{p_1}+W_2z^{p_2}$ with $p_1<p_2$,
    then $\dim(W)=\dim(W_1)+\dim(W_2)$ and
    \begin{align*}
      \dim(W^m)\geq \dim\left(\sum_{i=0}^mW_1^iW_2^{m-i}z^{ip_1+(m-i)p_2}\right)
      =\dim\left(\bigoplus_{i=0}^mW_1^iW_2^{m-i}z^{ip_1+(m-i)p_2}\right)
    \end{align*}
    where the ``='' holds since $ip_1+(m-i)p_2$ is a strictly decreasing function in variable $i$.
    Then
    \begin{align*}
      \dim(W^m)
      &\geq \sum_{i=0}^m\dim \left( W_1^iW_2^{m-i}z^{ip_1+(m-i)p_2}\right)= \sum_{i=0}^m\dim \left(W_1^iW_2^{m-i}\right)\\
      &=\frac12 \left(\sum_{i=0}^m\dim\left( W_1^iW_2^{m-i}\right)+
      \sum_{i=0}^m\dim\left( W_1^iW_2^{m-i}\right)\right)\\
      &\geq\frac12 \left(\sum_{i=0}^m\dim \left(W_1^i\right)+
      \sum_{i=0}^m\dim \left(W_2^{m-i}\right)\right)\\
      &\geq\frac12 \left(\sum_{i=0}^mc_n\dim \left(W_1\right)i^n+
      \sum_{i=0}^m c_n\dim \left(W_2\right)(m-i)^n\right)\ \  \text{(use SMC)}\\
      &=\frac12 c_n(\dim \left(W_1\right)+\dim \left(W_2\right))\left(\sum_{i=0}^m i^n\right)\\
      &\geq c'\dim \left(W\right)m^{n+1}
    \end{align*}
    where the last inequality follows from the inequality
    $\sum_{i=0}^m i^n>\frac{1}{n+1}m^{n+1}$ for $n\geq 0$.

    Now suppose that $W=W_1z^{p_1}+W_2z^{p_2}+\cdots+W_{l+1}z^{p_{l+1}}$ for some $l\geq 2$,
    and that
    \begin{align}
      \label{eqn-IH}
    \dim(U_l^m)\geq c'\dim(U_l) m^{n+1}, m\in \NNN
    \end{align}
    where
    $U_l=W_1z^{p_1}+W_2z^{p_2}+\cdots+W_lz^{p_l}$.
    Then
    \begin{align}
      \nonumber
      \dim(W^m)
      &\geq\dim\left(\sum_{i=0}^mU_l^{m-i}\left( W_{l+1}z^{p_{l+1}}\right)^i\right)\\
      \label{eqn-DirectSum}
      &\geq\dim\left(
      U_l^m+\sum_{i=1}^m(W_lz^{p_l})^{m-i}\left( W_{l+1}z^{p_{l+1}}\right)^i
      \right).
    \end{align}
    Note that $P_{n+1}$ has a natural $\NNN$-filtration
  \begin{align}
    \label{eqn-FiltrationLn}
  F_{0}\subseteq F_{1}\cdots\subseteq F_{m-1}\subseteq F_{m}\subseteq F_{m+1}\subseteq\cdots,
  \ m\in\NNN
  \end{align}
  where $F_m=\bigcup_{i\leq m}P_nz^i$.
  We have
  that $U_l^m\subseteq F_{mp_l}$,
    $\sum_{i=1}^m(W_lz^{p_l})^{m-i}\left( W_{l+1}z^{p_{l+1}}\right)^i\subseteq
    F_{mp_{l+1}}-F_{(m-1)p_l+p_{l+1}-1}$,
    and $mp_l\leq (m-1)p_l+p_{l+1}-1$.
    Thus the ``$+$'' of subspaces in Eq. (\ref{eqn-DirectSum}) is direct.
    Hence
    \begin{align*}
      \nonumber
      \dim(W^m)
      &\geq\dim\left(
      U_l^m\right)
      +\dim\left(\sum_{i=1}^m W_l^{m-i}W_{l+1}^iz^{(m-i)p_l+ip_{l+1}}
      \right)\\
      &\geq c'\dim(U_l) m^{n+1}
      +\sum_{i=1}^m \dim\left( W_{l+1}^i
      \right)\ \  \text{(use hypothesis Eq. (\ref{eqn-IH}))}\\
    &\geq c'\dim(U_l) m^{n+1}
      +\sum_{i=1}^m c_n\dim\left( W_{l+1}
      \right) i^n\ \  \text{(use SMC for $P_n$)}\\
    &\geq c'\dim(U_l) m^{n+1}
      +\frac{c_n}{n+1}\dim\left( W_{l+1}
      \right) m^{n+1}\\
      &\geq c'\dim(U_l) m^{n+1}
      +c'\dim\left( W_{l+1}
      \right) m^{n+1}\\
    &= c'\dim(W) m^{n+1}, \ \ m\in\NNN.
    \end{align*}
    Thus Eq. (\ref{eqn-W^mCase1}) holds as desired.

    Case 2.
    Return to the general form of $W$.
    By Lemma \ref{lemma-LeadTermMap} (using the natural filtration described as in Eq. (\ref{eqn-FiltrationLn})),
    $\gr(W)\supseteq \gr(V)\gr(a)=(V_n+kz)\gr(a)
    \supseteq V_n\alpha z^p+k\alpha z^{p+1}$
    where we assume $\gr(a)=\alpha z^p$ for some $\alpha\in P_n$ and $p\in\ZZ$.
    Let $W'=\gr(W)V_n\alpha z^p $.
    Then $W'$ has the form in Case 1 and $W'\subseteq (\gr(W))^2$.
    By Case 1 and Lemma \ref{lemma-LeadTermMap}, we have
    \begin{align*}
      \dim(W'^m)\geq c'\dim(W') m^{n+1}\geq c'\dim(\gr(W)) m^{n+1}
      =c'\dim(W) m^{n+1}.
    \end{align*}
    Thus
    \begin{align*}
     \dim(W^m)&\geq \dim((\gr(W))^m)\ \ \ \  \text{(use Lemma \ref{lemma-LeadTermMap})}\\
     &\geq \dim\left((\gr(W))^{2\cdot[\frac m2]}\right)
     \geq \dim\left((W')^{[\frac m2]}\right)\\
     &\geq c'\dim(W) {\left[\frac m2\right]}^{n+1}.
    \end{align*}
    where $\left[\frac m2\right]$ stands for the greatest integer that is not greater than $\frac m2$.
    It is easy to check that
    $
    {\left[\frac m2\right]}^{n+1}\geq \left(\frac m5\right)^{n+1}
    $
    for $m\geq 2$.
    Thus
    \[
    \dim(W^m)\geq c'\dim(W) \left(\frac m5\right)^{n+1}
    =\frac{c'}{5^{n+1}}\dim(W) m^{n+1}
    =c_{n+1}\dim(W) m^{n+1}, m\geq 2.
    \]
    It is easy to check that
    $\dim(W^m)\geq c_{n+1}\dim(W) m^{n+1}$ for $m=0,1$.
    That is, Eq. (\ref{eqn-W^m}) holds for all cases.
    Thus we have finished the proof.
\end{proof}

\end{document}